\documentclass[11pt]{amsart}
\usepackage{amsthm, amsmath,amsfonts, amssymb, enumerate, textcmds}
\usepackage[margin=1.4in]{geometry}
\usepackage{graphicx,tikz}
\usepackage{pinlabel}

\newtheorem{thm}{Theorem}[section]
\newtheorem{lem}[thm]{Lemma}
\newtheorem{prop}[thm]{Proposition}
\newtheorem{cor}[thm]{Corollary}

\theoremstyle{definition}
\newtheorem{defn}[thm]{Definition}
\newtheorem{remark}[thm]{Remark}
\begin{document}

\title[Infinite Energy Equivariant Harmonic Maps]{Infinite energy equivariant harmonic maps, domination, and anti-de sitter 3-manifolds}
\author{Nathaniel Sagman}

\address{Mathematics Dept., MC 253-37, Caltech, Pasadena, CA 91125} \email{nsagman@caltech.edu}

\maketitle
\begin{abstract}
  We generalize a well-known existence and uniqueness result for equivariant harmonic maps due to Corlette, Donaldson, and Labourie to a non-compact infinite energy setting and analyze the asymptotic behaviour of the harmonic maps. When the relevant  representation is Fuchsian and has hyperbolic monodromy, our construction recovers a family of harmonic maps originally studied by Wolf.
  
  We employ these maps to solve a domination problem for representations. In particular, following ideas laid out by Deroin-Tholozan, we prove that any representation from a finitely generated free group to the isometry group of a CAT$(-1)$ Hadamard manifold is strictly dominated in length spectrum by a large collection of Fuchsian ones. As an intermediate step in the proof, we obtain a result of independent interest: parametrizations of certain Teichm{\"u}ller spaces by holomorphic quadratic differentials. The main consequence of the domination result is the existence of a new collection of anti-de Sitter $3$-manifolds. We also present an application to the theory of maximal immersions into the Grassmanian of timelike planes in $\mathbb{R}^{2,2}$.
\end{abstract}

\begin{section}{Introduction}
Harmonic maps play a special role in the theory of geometric structures on manifolds. The existence results of Corlette, Donaldson, and Labourie link the purely algebraic data of a matrix representation of a discrete group to a geometric object--an equivariant harmonic map between manifolds--realising the prescribed transformations. In this paper we generalize their work to a non-compact setting and apply it to the study of domination between representations.

Let $\Gamma$ be a discrete group and for $k=1,2$ let $\rho_k : \Gamma \to \textrm{Isom}(X_k,g_k)$ be representations into the isometry groups of Riemannian manifolds $(X_k,g_k)$. A function $f:X_1\to X_2$ is $(\rho_1,\rho_2)$-equivariant if for all $\gamma\in \Gamma$ and $x\in X_1$, $$f(\rho_1(\gamma)\cdot x) = \rho_2(\gamma)\cdot f(x).$$ $\rho_1$ \textit{dominates} $\rho_2$ if there exists a $1$-Lipschitz $(\rho_1,\rho_2)$-equivariant map. The domination is \textit{strict} if the Lipschitz constant can be made strictly smaller than $1$. The translation length of an isometry $\gamma$ of a metric space $(X,d)$ is $$\ell(\gamma) = \inf_{x\in X}d(x,\gamma\cdot x).$$ $\rho_1$ dominates $\rho_2$ \textit{in length spectrum} if there is a $\lambda\in [0,1]$ such that $$\ell(\rho_2(\gamma))\leq \lambda \ell(\rho_1(\gamma))$$ for all $\gamma\in \Gamma$. This domination is strict if $\lambda<1$. From the definitions, (strict) domination implies (strict) domination in length spectrum.

Domination is essential to understanding complete manifolds locally modeled on $G=\textrm{PO}(n,1)_0$: a geometrically finite representation $\rho_1:\Gamma\to G$ strictly dominates $\rho_2:\Gamma\to G$ if and only if the $(\rho_1,\rho_2)$-action on $G$ by left and right multiplication is properly discontinous \cite{GK}. For $n=2$ these are the \textit{anti-de Sitter} (AdS) $3$-manifolds, and for $n=3$ we have the $3$-dimensional complex \textit{holomorphic-Riemannian} $3$-manifolds of constant non-zero curvature (see \cite{DZ} for details).

When $\Gamma$ is a closed surface group the landscape is well understood. In \cite{DT}, Deroin-Tholozan found that, given such a group $\Gamma$, any representation $\rho:\Gamma \to \textrm{Isom}(X,g)$ into the isometry group of a CAT($-1$) Hadamard manifold $X$ is strictly dominated by a Fuchsian representation unless $\rho$ stabilizes a totally geodesic plane of constant curvature $-1$ in which its action is Fuchsian. Gu{\'e}ritaud, Kassel, and Wolff proved the same result independently in the case $X=\mathbb{H}$ by realizing surface group representations geometrically as the holonomies of \textit{folded hyperbolic surfaces} (see \cite{GKW}). These results led to a new collection of closed AdS $3$-manifolds. In a follow-up paper \cite{T}, Tholozan showed that the representations from \cite{DT} exhaust the list of dominating pairs and (topologically) parametrized the deformation space of AdS structures as $$T(\Gamma)\times \textrm{Hom}^{nf}(\Gamma, \textrm{PSL}_2(\mathbb{R})).$$ Here $T(\Gamma)$ is the Teichm{\"u}ller space for $\Gamma$, and $\textrm{Hom}^{nf}(\Gamma, \textrm{PSL}_2(\mathbb{R}))$ is the subspace of the $\textrm{PSL}_2(\mathbb{R})$-deformation space for $\Gamma$ consisting of non-Fuchsian representations.

In this paper we focus on the case that $\Gamma$ is the fundamental group of a finite volume hyperbolic orbifold and examine $\textrm{PO}(2,1)_0\simeq \textrm{PSL}_2(\mathbb{R})$ geometry. By the Selberg lemma we reduce to the case where $\Gamma$ is a free group. Upon embedding $\Gamma$ into $\textrm{PSL}_2(\mathbb{R})$ as the fundamental group of a complete finite volume non-compact manifold $M$, we establish that any representation $\rho$ to the isometry group of a $\textrm{CAT}(-1)$ Hadamard manifold $X$ is dominated in length spectrum by a certain space of Fuchsian representations. By taking $X=\mathbb{H}$ this produces a collection of new AdS $3$-manifolds. The domination results in the present paper can be seen as the non-compact analogue of the work done in \cite{DT}. The space of Fuchsian representations strictly dominating a given non-Fuchsian representation is more complicated in the non-compact setting.

\begin{subsection}{Statement of main results}
Henceforth a manifold that is ``complete, finite volume" is implicitly understood to be non-compact. Owing to the classical Cartan-Hadamard theorem, a Riemannian manifold $(X,g)$ is \textit{Hadamard} if it is complete, simply connected, and non-positively curved. Such a manifold is $\textrm{CAT}(-\kappa)$, $\kappa\geq 0$, if all sectional curvatures are $\leq -\kappa$. See \cite{Bridson} for information on $\textrm{CAT}(-\kappa)$ metric spaces. When describing a fundamental group we suppress dependence on a basepoint. We often identify the fundamental group with the group of deck transformations without a change in notation. If $\Gamma$ acts isometrically on a Riemannian manifold and $\rho$ is a representation, an $(\textrm{id},\rho)$-equivariant map is simply called $\rho$-equivariant. The function $\Lambda: \mathbb{R}\to \mathbb{C}$ given by $$\Lambda(\theta)=(1-\theta^2) -i2\theta$$ will frequently appear in the paper. We record here that as $\theta$ increases from $-\infty$ to $\infty$, the complex argument of $\Lambda(\theta)$ decreases from $\pi$ to $-\pi$. 

All of the other relevant definitions and ambiguities will be discussed in later sections. 
Our first result generalizes the work of Corlette \cite{Corlette1}, Donaldson \cite{Donaldson}, and Labourie \cite{Labourieharm} and may also be regarded as an equivariant extension of \cite[Theorem 3.11]{Wolf}. Naturally, a portion of our analysis resembles that of Wolf.
\begin{thm}\label{first}
Let $M=\Tilde{M}/\Gamma$ be a complete finite volume hyperbolic surface and $(X,g)$ a Hadamard manifold. Let $\rho: \Gamma \to \textrm{Isom}(X,g)$ be a reductive representation. There exists a $\rho$-equivariant harmonic map $f:\Tilde{M}\to X$. If we assume $X$ is $\textrm{CAT}(-1)$, we may construct $f$ so that if $\gamma$ is a peripheral isometry and $\theta\in \mathbb{R}$, the Hopf differential $\Phi$ has the following behaviour at the corresponding cusp
\begin{itemize}
    \item if $\rho(\gamma)$ is parabolic or elliptic, $\Phi$ has a pole of order at most $1$ and
    \item if $\rho(\gamma)$ is hyperbolic, $\Phi$ has a pole of order $2$ with residue $$-\Lambda(\theta)\ell(\rho(\gamma))^2/16\pi^2.$$
\end{itemize}
If $\rho$ does not fix a point on $\partial_\infty X$, then $f$ is the unique harmonic map with these properties. If $\rho$ stabilizes a geodesic, then any other harmonic map with the same asymptotic behaviour differs by a translation along that geodesic.
\end{thm}
Regarding domination, the next theorem is the main result of this paper.

\begin{thm}\label{bigguy}
Let $M=\tilde{M}/\Gamma$ be a complete finite volume hyperbolic orbifold and $(X,g)$ a $\textrm{CAT}(-1)$ Hadamard manifold. Let $\rho:\Gamma\to \textrm{Isom}(X,g)$ be any representation.  There exists a geometrically finite representation $j_M$ dominating $\rho$ in length spectrum. If $\rho$ is reductive, then $j_M$ dominates $\rho$ in the traditional sense. There is a family of convex cocompact Fuchsian representations strictly dominating $j_M$. Given a peripheral isometry $\gamma$,
\begin{itemize}
    \item if $\rho(\gamma)$ is not hyperbolic, then $j_M(\gamma)$ is parabolic and
    \item if $\rho(\gamma)$ is hyperbolic, $j_M(\gamma)$ is hyperbolic with the same translation length.
\end{itemize}
\end{thm}

 In general $j_M$ will not strictly dominate $\rho$. This will be discussed in detail in Section 6. If $X=\mathbb{H}$ and $\rho$ is Fuchsian with no elliptic monodromy it will follow from the proof that $j_M=\rho$. For holonomy representations of closed surfaces, Thurston observed in \cite[Proposition 2.1]{Thurston} that strict domination contradicts the Gauss-Bonnet theorem and is therefore impossible. 

Most of the proof of Theorem \ref{bigguy} is devoted to constructing $j_M$. To upgrade to a strictly dominating representation we perform a \textit{strip deformation}, a procedure introduced by Thurston \cite{Thurston} and further developed in \cite{margulis}.
 
 Setting $X=\mathbb{H}$ in Theorem \ref{bigguy}, from \cite[Theorem 1.8]{GK} we obtain:

\begin{thm}\label{notmain}
Let $M=\tilde{M}/\Gamma$ be a complete finite volume hyperbolic orbifold and  $\rho:\Gamma\to \textrm{PSL}_2(\mathbb{R})$ any representation. Then there is a Fuchsian representation $j_M$ dominating $\rho$ and a family of convex cocompact representations $(j_M^\alpha)$ strictly dominating $j_M$ such that $$(\rho\times j_M^\alpha)(\Gamma)\subset \textrm{PSL}_2(\mathbb{R})\times \textrm{PSL}_2(\mathbb{R})$$ admits a properly discontinuous action on $\textrm{PSL}_2(\mathbb{R})$ preserving the Lorentz metric of constant curvature $-1$. If $\gamma\in \Gamma$ is elliptic and $\rho(\gamma)$ has smaller order than $\gamma$, then the action is torsion free as well. Consequently there exists a geometrically finite AdS $3$-manifold Seifert-fibered over $\mathbb{H}/j_M^\alpha(\Gamma)$.
\end{thm}

Note that if $M$ is a manifold, the torsion condition always holds.

As an intermediate step in the proof of Theorem \ref{bigguy}, we obtain a result of independent interest. Let $M$ be a complete finite volume hyperbolic surface with $n$ punctures and let $T(M,p_1,\dots, p_{d_1}, \ell_{d_1+1},\dots, \ell_{n})$ denote the subspace of the Fricke-Teichm{\"u}ller space of $M$ consisting of holonomies of hyperbolic surfaces with $d_1$ ordered punctures and $d_2$ ordered geodesic boundary components of length $\ell_{d_1+1},\dots, \ell_{d_2}>0$. Let $(\theta_k)_{k=d_1+1}^n\subset \mathbb{R}$ and $P:=(\ell_k,\theta_k)_{k=d_1+1}^{n}$. Denote by $Q(M,P)$ the space of holomorphic quadratic differentials on $M$ with poles of order at most one at the punctures corresponding to cusps and poles of order $2$ with residue $$-\Lambda(\theta_k)\ell_k^2/16\pi^2$$
for each puncture labelled by $\ell_k$. From the results in \cite{Wolf}, for each point in $T(M,p_1,\dots, p_{d_1}, \ell_1,\dots, \ell_{d_2})$, there is a unique homotopic harmonic diffeomorphism $h_f: M\to S$ whose Hopf differential lives in $Q(M,P)$.

\begin{thm} \label{frick} 
Let $M$ be a finite volume hyperbolic surface. The map $$\Psi: T(M,p_1,\dots, p_{d_1}, \ell_1,\dots, \ell_{d_2})\to Q(M,P)$$ given by $[S,f]\mapsto \textrm{Hopf}(h_f)$ is a homeomorphism.
\end{thm}

We expect the above result is known to experts, but could not find a proof in the literature. Hence we supply our own. The parametrization of the Teichm{\"u}ller space of a closed surface by holomorphic quadratic differentials goes back to Sampson, Schoen-Yau, and Wolf (see \cite{wolf1} for the full result). The case of Teichm{\"u}ller spaces of punctured surfaces, corresponding to differentials with a pole of order at most $1$, was completed by Lohkamp \cite{Lohkamp}.  In \cite{Wild}, Gupta parametrized \textit{wild Teichm{\"u}ller spaces} by certain equivalence classes of holomorphic differentials with poles of order at least $3$. Theorem \ref{frick} thus completes a description of the space of meromorphic quadratic differentials over a Riemann surface in terms of harmonic diffeomorphisms.

 In \cite{AL}, Alessandrini and Li explore the domination phenomena via Higgs bundles and geometric structures. Building on work from Baraglia's thesis \cite{bar}, they observe a link between AdS structures and maximal immersions into Grassmanians of timelike planes. Utilizing this machinery, we observe new phenomena that do not occur for closed surfaces.

\begin{cor} \label{minimal}
Let $M=\tilde{M}/\Gamma$ be a complete finite volume hyperbolic surface and let $\rho:\Gamma\to \textrm{SL}_2(\mathbb{R})$ be any non-Fuchsian reductive representation. Then there exists a Fuchsian representation $j_M$ and a circle bundle $p:U\to M$ that admits an $AdS^3$ geometric structure with holonomy $(j_M\otimes \rho)_*\circ p_*$, but such that the projectivization of $(j_M,\rho)$ does not act properly discontinuously on $\textrm{PSL}_2(\mathbb{R})$. Moreover, there exists a $(j_M,\rho)$-equivariant maximally immersed surface in the Grassmanian of timelike planes in $\mathbb{R}^{2,2}$.
\end{cor}

We end this subsection by presenting a quick corollary of Theorem \ref{bigguy}, unrelated to the rest of the paper. When $X$ is a CAT($-1$) Hadamard manifold and $\rho:\Gamma \to \textrm{Isom}(X,g)$ is geometrically finite, the \textit{limit set} of $\rho(\Gamma)$ is the set of limit points of $\Gamma \cdot z$ in $\partial_\infty X$ for a fixed point $z$ in $X$. It is a standard exercise to confirm that this does not depend on the point $z$. When $X=\textrm{PSL}_2(\mathbb{R})$ and $\rho$ is Fuchsian, the limit set is either the full circle $\partial_\infty \mathbb{H}$ or a Cantor set. The \textit{critical exponent} $\delta(\rho)$ is the smallest constant $s$ such that the Poincar{\'e} series $$\sum_{\gamma \in \Gamma} e^{-sd(z,\rho(\gamma)\cdot z)}$$ converges, and it coincides with the Hausdorff dimension of the limit set  (see \cite{MCo} for a proof). The analogue of the following result is known for closed surfaces and is observed in \cite{DT}, but to the author's knowledge it is new in our context.

\begin{cor} 
Let $M=\tilde{M}/\Gamma$ be a complete finite volume hyperbolic orbifold and $(X,g)$ a $\textrm{CAT}(-1)$ Hadamard manifold. Let $\rho:\Gamma\to \textrm{Isom}(X,g)$ be a geometrically finite representation. There is a Fuchsian representation $j_M$ such that the Hausdorff dimension of the limit set of $\rho(\Gamma)$ is bounded below by that of $j_M$. $j_M$ has the following property around a peripheral $\gamma$:
\begin{itemize}
    \item if $\rho(\gamma)$ is not hyperbolic, then $j_M(\gamma)$ is parabolic and
    \item if $\rho(\gamma)$ is hyperbolic, then $j_M(\gamma)$ is hyperbolic with $\ell(j_M(\gamma))=\ell(\rho(\gamma))$.
\end{itemize}
\end{cor} 

The Hausdorff dimension can be estimated and sometimes fully understood from the monodromy around the punctures. For instance, if $M$ is a pair of pants and $\rho$ takes the cuffs to isometries with lengths $a,b,c>0$, then the Hausdorff dimension of the limit set of $j_M$ occurs as a zero of a certain \textit{Selberg zeta function} $$Z_{a,b,c}(s)=\prod_{\gamma\in\Gamma}\prod_{m=0}^\infty \Big (1-e^{-(s+m)\ell(\gamma)} \Big ).$$ The map $\gamma\mapsto \ell(\gamma)$ is determined entirely by $a,b,c$. These zeroes can be computed efficiently (see \cite{pants} for details).
\end{subsection}
\begin{subsection}{Anti-de Sitter 3-manifolds}
Let $\mathcal{B}$ be a non-degenerate symmetric billinear form on $\mathbb{R}^{n}$ with signature $(n-2,2)$. The \textit{Anti-de Sitter space} is $$\textrm{AdS}^n = \{x\in \mathbb{R}^n: \mathcal{B}(x,x)=-1 \}/\{\pm 1\}.$$ The tangent space to any $[x]\in \textrm{AdS}^n$ identifies with the $\mathcal{B}$-orthogonal complement of the subspace generated by $x$. The restriction of $\mathcal{B}$ to each tangent space is a billinear form of signature $(n-1,1)$. This endows $\textrm{AdS}^n$ with the structure of a Lorentzian manifold with constant sectional curvature $-1$ across all non-degenerate $2$-planes. A tangent vector $v\in T_x \textrm{AdS}^n$ is \textit{timelike}, \textit{spacelike}, or \textit{lightlike} if $\mathcal{B}(v,v)<0$, $\mathcal{B}(v,v)>0$, or $\mathcal{B}(v,v)=0$ respectively. A complete \textit{anti-de Sitter n-manifold} is a quotient of $\textrm{AdS}^n$ by a properly discontinuous subgroup of Lorentzian isometries. AdS manifolds play a deep role in mathematical physics and are also of independent mathematical interest.

In dimension $n=3$, $\textrm{AdS}^3$ identifies isometrically with the Lie group $\textrm{PSL}_2(\mathbb{R})$ equipped with the Lorentz metric induced by (a constant multiple of) the Killing form on the Lie algebra. The group of orientation and time-preserving isometries is $\textrm{PSL}_2(\mathbb{R})\times \textrm{PSL}_2(\mathbb{R})$ acting via left and right multiplication: $$((g,h), x)\mapsto (g,h)\cdot x = gxh^{-1}.$$ The data of a subgroup of isometries is equivalent to that of two representations into $\textrm{PSL}_2(\mathbb{R})$. AdS $3$-manifolds are studied in \cite{Klingler},  \cite{KR}, \cite{Sa}, \cite{Ka}, \cite{DT}, \cite{GKW}, \cite{margulis}, and \cite{DGK3} among other sources.
Scott \cite{Scott} proved that many closed $3$-manifolds can be modelled on Thurston's $6^{th}$ geometry. Such manifolds also admit an $\textrm{AdS}^3$ structure that is called standard. Goldman found more examples by deforming the standard structures \cite{nonstandard}. Salein observed in his thesis that domination is a sufficient criterion for proper discontinuity. By a complex analytic argument he produced more examples and showed that for a closed surface group the projection from $$\textrm{Hom}
(\Gamma,\textrm{PSL}_2(\mathbb{R}))^2\supset \{\textrm{properly discontinuous pairs }(j,\rho)\}$$ to $\textrm{Rep}^{nf}(\Gamma, \textrm{PSL}_2(\mathbb{R}))$ intersects every non-extremal Euler class. Kassel proved domination is necessary \cite{Ka} and along with Gu{\'e}ritaud extended the result to a higher dimensional setting in \cite{GK}. These two have the definitive result.

\begin{thm} \label{main}
(Gu{\'e}ritaud, Kassel) Let $j,\rho: \Gamma\to PSO(n,1)$ be representations of a discrete group such that $j$ is geometrically finite. Then $(j,\rho)$ acts properly discontinuous on $PSO(n,1)$ by left and right multiplication if and only if the infimum of Lipschitz constants of $(j,\rho)$-equivariant maps is $<1$.
\end{thm}

 As discussed earlier, the remaining examples of closed AdS $3$-manifolds came from \cite{DT} and \cite{GKW}. In their study of \textit{Margulis spacetimes} in \cite{margulis}, Danciger, Gu{\'e}ritaud, and Kassel parametrized the subspace of representations into $\textrm{PSL}_2(\mathbb{R})$ strictly dominating $\rho$ in the case $\Gamma$ is free and $\rho$ is Fuchsian.

Note that one could similarly have representations $(\rho,j)$ with $j$ geometrically finite and $\rho$ non-Fuchsian that lead to properly discontinuous actions. The resulting AdS $3$-manifolds are isomorphic if we swap $\rho$ and $j$. 
The AdS $3$-manifolds obtained via such a pair $(j,\rho)$ are Seifert-Fibered over $\mathbb{H}/j(\Gamma)$, and each fiber is a timelike geodesic.

\end{subsection}

\begin{subsection}{Outline of paper and strategy of proof}
In the next section we introduce the relevant definitions and notations in the representation theory of discrete groups. We also reduce Theorem \ref{bigguy} to the case where $M$ is a manifold. In Section 3 we discuss harmonic maps and energy. We prove that the energy of an equivariant harmonic map is bounded above by that of a special harmonic diffeomorphism of the disk with the same Hopf differential. As is standard in this field, we argue via an analysis of the Bochner formula. This estimate is a central technical results of this paper, and is instrumental in proving Theorem \ref{bigguy}. 
 
  In Section 4 we prove Theorem \ref{frick} using classical techniques from the theory of harmonic maps. Section 5 is devoted to the proof of Theorem \ref{first}. Infinite energy harmonic maps were constructed for some special cases in \cite{Wolf}, \cite{Simpson}, \cite{JZ}, and \cite{MK}. Consequently, there is nothing truly novel in the proof of the general existence result--it is an amalgamation of known ideas. The real work is done in studying the behaviour and uniqueness of the harmonic maps. We combine the energy estimate from Section 3 with Theorem \ref{frick} to control the energy locally, as well as a distance comparison to a special non-harmonic map to understand the directions in which our map should expand and contract.

In Section $6$ we take an equivariant harmonic map $f$ from Theorem \ref{first} and choose a harmonic diffeomorphism $h$ from $M$ to the convex core of some geometrically finite hyperbolic surface $N$ that has the same Hopf differential as $f$. From our energy estimates $f\circ h^{-1}$ is $1$-Lipschitz and intertwines $\rho$ with the holonomy of $N$.  The compact analogue of this idea was laid out in \cite{DT} and expounded upon in \cite{AL}. We introduce strip deformations to strictly dominate the holonomy of $N$, completing the proof of Theorem \ref{bigguy}. In the final section we briefly discuss geometric structures and maximal immersions into Grassmanians.
\end{subsection}

\begin{subsection}{Other recent work}
 Shortly after a preprint of this paper was posted to the arXiv, Gupta-Su proved the same domination result for representations to $\textrm{PSL}_2(\mathbb{C})$ \cite{GS}. Their proof is different: they straighten the pleated plane determined by the Fock-Goncharov coordinates associated to a framed representation, and then use strip deformations.  
\end{subsection}

\begin{subsection}{Acknowledgements}
It is a pleasure to thank my advisor, Professor Vladimir Markovic, for his support, insight, patience, and guidance. I would also like to thank Qiongling Li and Peter Smillie for their interest and helpful conversations, as well as Fran\c{c}ois Gu{\'e}ritaud for graciously answering some questions over email. Finally, I would like to thank my friend Arian Jadbabaie for helping create Figure 1 on Inkscape.
\end{subsection}

\end{section}

\begin{section}{Representations of discrete groups}
\begin{subsection}{Isometries of Hadamard manifolds}
Throughout, let $(X,g)$ be a $\textrm{CAT}(0)$ Hadamard manifold. We will denote by $\partial_\infty X$ the visual boundary of $X$ endowed with its natural topology (see \cite{Bridson} for details). Isometries of $X$ extend to homeomorphisms of $X\cup \partial_\infty X$. When discussing isometries, we appeal to the standard classification.
\begin{defn}
An isometry $\gamma \in \textrm{Isom}(X,g)$ is
\begin{itemize}
    \item \textit{elliptic} if $\ell_X(\gamma)=0$ and it has a fixed point inside $X$,
    \item \textit{parabolic} if $\ell_X(\gamma)=0$ and it has no fixed points inside $X$, or
    \item \textit{hyperbolic} if $\ell_X(g)>0$.
\end{itemize}
\end{defn}
A hyperbolic isometry preserves a unique geodesic axis on which it acts by linear translation. On any point $x$ of this axis, $d(x,\gamma \cdot x)=\ell(\gamma)$.

 Let $\Gamma\subset \textrm{PSL}_2(\mathbb{R})$ be such that $M=\mathbb{H}/\Gamma$ is a complete  hyperbolic orbifold. A \textit{cusp neighbourhood} is a region surrounding a puncture of $M$ such that for some $a,\tau>0$ it identifies isometrically with $$U(\tau):=\{z=x+iy: (x,y)\in [0,\tau]\times [a,\infty)\}/\langle z\mapsto z+\tau\rangle$$ equipped with the hyperbolic metric $|dz|^2/y^2$. An isometry in $\Gamma$ is \textit{peripheral} if upon identifying $\Gamma$ with the fundamental group its conjugacy class corresponds to a curve surrounding a puncture. We will also say such curves are peripheral. A representation $\rho: \Gamma\to \textrm{Isom}(X,g)$ has elliptic, parabolic, or hyperbolic \textit{monodromy} around a puncture if the image of the peripheral conjugacy class is of that type.
 
 \begin{defn}
  The \textit{convex core} $C(M)$ of $M$ is the quotient of the convex hull of the limit set of $\Gamma$ by the action of $\Gamma$.
 \end{defn}
 
It enjoys the property that the inclusion $C(M)\to M$ is a homotopy equivalence. The convex core is the union of a compact orbifold and a number of cusp regions for each puncture in $M$. It may have a number of geodesic boundary components. $M$ can be recovered from $C(M)$ by attaching infinite funnels along these components.
 
 \begin{defn}
  A discrete representation $j : \Gamma\to \textrm{PSL}_2(\mathbb{R})$ is \textit{convex cocompact } (or \textit{geometrically finite}) if the convex core of $j(\Gamma)$ is compact (or finite volume).
 \end{defn}
 
 The representation is convex cocompact precisely if the convex core has no cusps. For an exposition of the general theory of geometrically finite representations in $\textrm{PSO}(n,1)$ we refer the reader to \cite{GK}. It is unique to dimension $n=2$ that $j(\Gamma)$ is finitely generated if and only if $\mathbb{H}/j(\Gamma)$ is geometrically finite.
 
 A subgroup of $\textrm{Isom}(X,g)$ is \textit{parabolic} if all elements have a common fixed point in $\partial_\infty X$. In this paper, a \textit{flat} in $X$ is a subspace isometric to $\mathbb{R}^n$ with its Euclidean metric.
 
 \begin{defn}
   A representation $\rho:\Gamma\to \textrm{Isom}(X,g)$ is \textit{reductive} if $\rho(\Gamma)$ is not contained in a parabolic subgroup or if it stabilizes some totally geodesic flat $Y$.
 \end{defn}
 
 When $M$ is a surface and $X$ is $\textrm{CAT}(-1)$, this condition is equivalent to $\rho$ not fixing a point on the boundary $\partial_\infty X$ or stabilizing a geodesic. In the more special case that $X$ is a rank $1$ symmetric space of non-compact type, so that $\textrm{Isom}(X,g)$ is a linear algebraic group, this is equivalent to the usual notion that the Zariski closure of $\rho(\Gamma)$ is reductive.  
\end{subsection}

\begin{subsection}{Representations and flat bundles}
Let $M=\Tilde{M}/\Gamma$ be a connected manifold. We consider the set of representations into $\textrm{Isom}(X,g)$ modulo conjugation $$\textrm{Hom}(\Gamma, \textrm{Isom}(X,g))/\Gamma.$$ Out of any representation $\rho: \Gamma \to \textrm{Isom}(X,g)$ we can construct a flat fiber bundle $$X_\rho:= X\times_\rho \Tilde{M}:=\{(x,s)\in X\times \Tilde{M}: (x,s)\sim (\gamma x, \rho(\gamma) s)\}.$$ Conjugate representations give rise to isomorphic bundles. Conversely, any flat bundle over $X$ has a well-defined holonomy representation $\Gamma\to \textrm{Isom}(X,g)$. This construction describe a bijection between the space of representations up to conjugacy and the space of flat bundles up to gauge equivalence (see \cite{labourie} for a proof and more details). 

Global sections always exist because $X$ is contractible. Under this correspondence, taking the pullback bundle with respect to the universal covering $\Tilde{M}\to M$ shows that sections of $X_\rho$ are equivalent to $\rho$-equivariant maps from $\Tilde{M}\to X$. We will pass back and forth between these two perspectives.
\end{subsection}
\begin{subsection}{Optimal Lipschitz constants}
Given $\Gamma$ discrete, $\rho:\Gamma\to \textrm{Isom}(X,g)$, and $j:\Gamma\to \textrm{PSL}_2(\mathbb{R})$ geometrically finite, we set $$C(j,\rho):=\inf \textrm{Lip}(f),$$ where the infimum is taken over the family of all $(j,\rho)$-equivariant Lipschitz maps. The theorem below is Theorem 1.8 in \cite{GK}.

\begin{thm} \label{gkthm}
(Gu{\'e}ritaud, Kassel) Let $\Gamma$ be a discrete group and $\rho,j:\Gamma\to \textrm{PSL}_2(\mathbb{R})$ two representations with $j$ geometrically finite. Then $C(j,\rho)<1$ if and only if $$C(j,\rho)':= \sup \frac{\ell(\rho(\gamma))}{\ell(j(\gamma))}<1,$$ unless $\rho$ has exactly one fixed point on $\partial_\infty \mathbb{H}$ and there exists a $\gamma\in \Gamma$ such that $j(\gamma)$ is parabolic and $\rho(\gamma)$ is not elliptic.
\end{thm}

\begin{remark}
As we will see later, equivariant harmonic maps only exist for reductive representations. To dominate non-reductive representations we would like to use a version of Theorem \ref{gkthm} that holds for variable curvature. The result and the proof of Theorem \ref{gkthm} do not directly transfer. While the ideas for a suitable reformulation and proof lie in \cite{GK}, this is not the focus of the present paper. Hence, for the non-reductive case we settle for length spectrum domination, although we expect the full domination result to be true. From the theorem above, non-reductive representations still lead to AdS $3$-manifolds, which is the most important application. 
\end{remark}

Now suppose $M=\tilde{M}/\Gamma$ is a complete finite volume hyperbolic orbifold. By the Selberg lemma $\Gamma$ admits a finite index torsion free normal subgroup $\Gamma_0$. The quotient $\tilde{M}/\Gamma_0$ is a complete finite volume hyperbolic manifold. We close this section with a lemma that reduces Theorem \ref{bigguy} to the case of hyperbolic manifolds.

\begin{lem} \label{lem:selberg}
  Let $\Gamma$ be a discrete group and $\Gamma_0\subset \Gamma$ a finite index normal subgroup. Let $\rho: \Gamma \to \textrm{Isom}(X,g)$ and $j:\Gamma\to \textrm{PSL}_2(\mathbb{R})$ be representations and let $\rho_0$ and $j_0$ be their restrictions to $\Gamma_0$. Then $C(j,\rho)=C(j_0,\rho_0)$.
\end{lem}

This is essentially done in \cite{GK}, although the authors prove something more general and restrict to the case $X=\mathbb{H}^n$. Our proof relies on a lemma from \cite{GK}. 

\begin{lem} \label{gklemma}
   Let $I$ be any countable index set and $\alpha=(\alpha_i)_{i\in I}\subset \mathbb{R}$ a sequence summing to $1$. Given $p\in K\subset \mathbb{H}$ and $f_i: K\to X$, $i\in I$ such that $$\sum_{i\in I} \alpha_i d(f_1(p),f_i(p))<\infty,$$ the map $$f:=\sum_{i\in I}\alpha_i f_i \hspace{1mm} , \hspace{1mm} x\mapsto \textrm{argmin}\Big \{ p'\in X : \sum_{i\in I} \alpha_i d(p',f_i(x))<\infty \Big \}$$ is well-defined and satisfies $$\textrm{Lip}_x(f)\leq \sum_i \alpha_i \textrm{Lip}_x(f_i) \hspace{1mm} , \hspace{1mm} \textrm{Lip}_Y(f)\leq \sum_i \alpha_i \textrm{Lip}_Y(f_i).$$ If each $f_i$ is equivariant with respect to a pair of representations then so is $f$.

\end{lem}

The authors give a proof for $X=\mathbb{H}^n$ but the proof only uses the fact that $\mathbb{H}^n$ is a $\textrm{CAT}(0)$ metric space.

\begin{proof}[Proof of lemma \ref{lem:selberg}]
If no $(j',\rho')$-equivariant maps exist there is nothing to prove, so assume otherwise. The inequality $C(j',\rho')\leq C(j,\rho)$ is obvious because any $(j,\rho)$-equivariant map is $(j',\rho')$-equivariant. As for the other inequality, write $$\Gamma=\coprod_{i=1}^r\gamma_i\Gamma_0$$ for some collection of coset representatives $\gamma_i$. Let $f$ be a $(j',\rho')$-equivariant map. Notice that for any $\gamma\in \Gamma$, the map  $$f_\gamma: =\rho(\gamma)^{-1}\circ f \circ j(\gamma)$$ depends only on the coset $\gamma\Gamma_0$. Indeed, suppose we are given $\gamma_1,\gamma_2\in \Gamma$ such that $\gamma_1\gamma_2^{-1}\in \Gamma_0$. For $x\in \mathbb{H}$ let $y=j(\gamma_2)^{-1}x$. Then $$f_{\gamma_1}(x) = \rho(\gamma_1)^{-1}\circ f(j(\gamma_1\gamma_2^{-1})y) = \rho(\gamma_2)^{-1}\circ f(y)=f_{\gamma_2}(x).$$ By Lemma \ref{gklemma} the map $$f':= \sum_{i=1}^r \frac{1}{r}\cdot f_{\gamma_i}$$ satisfies $$\rho(\gamma)^{-1}\circ f'\circ j(\gamma)= \sum_{i=1}^r \frac{1}{r}\cdot f_{\gamma\gamma_i}= f'$$ since the sum in the middle is just a rearrangement of the sum describing $f'$. By Lemma \ref{gklemma} again we have $\textrm{Lip}(f')\leq \textrm{Lip}(f)$. Taking $\textrm{Lip}(f)\to C(j',\rho')$, the lemma follows.
\end{proof}
\end{subsection}
\end{section}
\begin{section}{Harmonic maps}
Throughout the paper, we use the letter $A$ to denote some large uniform constant. In the course of a proof, $A$ may grow larger and we will not make this explicit in our notation. At times we write $A_z$ to highlight dependence on some quantity $z$. For functions $f_1,f_2$ we use the convention $f_1\lesssim f_2$ to mean $f_1\leq Af_2$.
\begin{subsection}{Definitions and basic properties}
Let $f:(M,g_0)\to (X,g)$ be a $C^2$ map between Riemannian manifolds. The derivative $df$ defines a section of the endomorphism bundle $T^*M\otimes f^*TX$. This bundle inherits a natural Riemannian metric, and the energy density $e(f)$ is defined to be $$e(f)= \frac{1}{2}||df||_{T^*M\otimes f^*TX}^2=\frac{1}{2}\textrm{tr}_{g_0} f^* g.$$

\begin{defn}
$f$ as above is said to be $\textit{harmonic}$ if for every  $V\subset \subset M$ it is a critical point of the energy functional $$E_V(f) := \int_V e(f) dv_{g_0}$$ 
\end{defn}
Here $dv_{g_0}$ denotes the volume form. Equivalently, $f$ is harmonic if it solves the Euler-Lagrange equation $$\textrm{tr}_{g_0}\nabla df = 0,$$ where $\nabla$ is the Levi-Civita connection on $T^*M\otimes f^*TX$. From elliptic regularity, the harmonic maps in this paper are $C^\infty$.

Let $M=\tilde{M}/\Gamma$ and take a representation $\rho: \Gamma\to \textrm{Isom}(X,g)$. Let $f:\tilde{M}\to X$ be a $C^2$ $\rho$-equivariant map. Since $\rho$ is acting by isometries, the energy density descends to a well-defined function on $M$.

\begin{defn}
A $\rho$-equivariant map $f: \tilde{M}\to X$ is \textit{harmonic} if for every $V\subset \subset M$ it is a critical point along variations of equivariant maps for the twisted energy functional $$E_V(f) :=\int_V e(f) dv_{g_0}.$$ 
\end{defn}

Equivariant harmonic maps inherit all local properties of harmonic maps. When we state a result for harmonic maps it will be implicitly understood to hold for equivariant harmonic maps as well. 

Now we specialize: assume henceforth $M$ is a Riemann surface and $g_0(z)=\sigma(z)|dz|^2$. 

\begin{defn}
A \textit{holomorphic quadratic differential} is a section of the second symmetric power of the holomorphic cotangent bundle. In a local complex coordinate $z$ a quadratic differential is simply a tensor of the form $$\phi(z)dz^2,$$ where $\phi$ varies holomorphically in $z$. 
\end{defn}

A holomorphic quadratic differential $\Phi=\phi(z)dz^2$ defines a singular metric via $|\phi(z)|dz^2$. In the sequel we will call it the $\Phi$-metric. It is flat off the zeroes, where it has cone points.  Away from the zeroes one can choose holomorphic coordinates so that $$\Phi=dz^2$$ Writing $z=x+iy$, the ``horizontal" direction is $x$ and the ``vertical" is $y$. If $f:M\to X$ is harmonic the pullback metric is $$f^*g = e(f)\sigma(z)dzd\overline{z}+\Phi dz^2 + \overline{\Phi}d\overline{z}^2.$$ The $(2,0)$ component $\Phi$ is a quadratic differential, and by harmonicity of $f$ one can check it is holomorphic. $\Phi=:\textrm{Hopf}(f)$ is called the \textit{Hopf differential}.

There are two important quantities associated with a harmonic map from a Riemann surface: the holomorphic energy $H$ and anti-holomorphic energy $L$. Upon complexifying the tangent bundle of $M$, the derivative decomposes into a holomorphic and anti-holomorphic component as $$df = df_z+df_{\overline{z}}.$$ We put $H(f):=||df_z||_{T^*M\otimes f^*TX}^2$, $L(f):=||df_{\overline{z}}||_{T^*M\otimes f^*TX}^2$, so that $e(f) = H(f)+L(f)$. We highlight that $$||\Phi||^2:=|\Phi|_e^2/\sigma^2= H(f)L(f),$$ where $|\cdot|_e$ denotes the Euclidean norm on the space of holomorphic differentials. The Jacobian $J(f)$ satisfies $$J(f)=H(f)-L(f).$$ We also recall the Bochner formula. When $H(f)> 0$ and the metric $f^*g$ is non-degenerate, $$\Delta \log H(f) = -2\kappa(f^*g) H(f) + 2\kappa(f^*g)\frac{||\Phi||^2}{H(f)}+2\kappa(g_0),$$ where $\kappa(\cdot)$ denotes sectional curvature of a metric on $M$ at a point and $$\Delta=\Delta_{\sigma} := \frac{1}{\sigma(z)}\frac{\partial^2}{\partial z \partial \overline{z}}$$ is the Laplacian with respect to the metric $\sigma$. More information in much more general context can be found in \cite{ogtext}, \cite{heat}, and \cite{schoenyau}.
\end{subsection}
\begin{subsection}{Important results for harmonic maps}
In their seminal work \cite{ES}, Eells and Sampson pioneered the heat flow method to prove existence of harmonic maps between closed manifolds when the target has non-positive sectional curvature. Since then there has been a huge amount of progress and far reaching generalizations (see \cite{heat} for instance).

Donaldson implemented the heat flow in \cite{Donaldson} to prove existence of equivariant harmonic maps when $M$ is a closed Riemann surface and $X=\mathbb{H}^3$. This method was used in more general contexts by Corlette \cite{Corlette1}, Labourie \cite{Labourieharm}, Jost-Yau \cite{jostyau}, and Corlette again \cite{Corlette2}. We compile some of the results into one:

\begin{thm} \label{classic} (Corlette, Donaldson, Labourie, Jost-Yau) Suppose $M$ is a complete Riemannian manifold possibly with boundary, $\Gamma\simeq \pi_1(M)$, $X$ is a $\textrm{CAT}(-1)$ Hadamard manifold, and $\rho:\Gamma\to X$ is a reductive representation. If there exists a $\rho$-equivariant map with finite energy (with equivariant boundary values if $\partial M \neq \emptyset$), then there exists an equivariant harmonic map (with the same boundary values).

\end{thm}

More details and cases are contained in the sources above. Finally we collect some technical lemmas that will be used in the paper. The first result follows from the definitions.

\begin{lem}
   Let $(N,g_0')$, $(M,g_0)$ be Riemann surfaces and $(X,g)$ a Riemannian manifold. If $\psi:N\to M$ is conformal and $f:M\to X$ is harmonic then $f\circ\psi$ is harmonic.
\end{lem}

For points $x,y\in X$ we let $d(x,y)$ denote their distance with repsect to the metric $g$. The next result is a consequence of the Hessian comparison theorem. 

\begin{lem}
  If $f_1,f_2:\tilde{M}\to X$ are harmonic maps then the function $$p\mapsto d^2(f_1(p),f_2(p))$$ is $C^\infty$ and subharmonic.
\end{lem}

We will even see a more precise result later on. The final result is referred to as Cheng's Lemma. For the original proof see \cite{chenglemma}.

\begin{lem}
  Let $\tilde{M}$ be a complete simply connected surface with a metric $g_0$ of curvature pinched between $-b^2$ and $0$ for some $b>0$, and let $X$ be a Hadamard manifold. For any $z\in \tilde{M}$ and $r>0$, let $f:B_{g_0}(z,r)\to X$ be a harmonic map whose image lies in a ball of radius $R_0$. Then $$||df||_{T^*\tilde{M}\otimes f^*TX}\lesssim R_0 \cdot\frac{1+br_0}{r_0}.$$
\end{lem}
\end{subsection}

\begin{subsection}{Energy of harmonic maps}
A harmonic map has finite energy if the \textit{total energy} $E_M(f)=:E(f)$ is finite. The harmonic maps of Theorem \ref{first} will not necessarily have $E(f)<\infty$. We give a precise criterion.

\begin{prop} \label{tame}
 If $M=\tilde{M}/\Gamma$ is a complete finite volume hyperbolic surface, $(X,g)$ is a $\textrm{CAT}(-1)$ Hadamard manifold, and $\rho: \Gamma \to \textrm{Isom}(X,g)$ is a representation, then a finite energy $\rho$-equivariant map exists if and only if $\rho$ has no hyperbolic monodromy.
\end{prop}

Before we begin we build a new metric that will be used throughout the paper.  Label the cusp neighbourhoods $C_1,\dots, C_n$. Take collar neighbourhoods $U_k$ of each $\partial C_k$ inside $M\backslash C_k$ and consider the metric on $M$ that agrees with the hyperbolic metric on $M\backslash (\cup_k C_k)$ and is flat on each $C_k\cup U_k$. Then interpolate on a neighbourhood of $\partial U_k\backslash \partial C_k$ that does not touch $\partial C_k$ to a smooth non-positively curved metric $\sigma'$, conformally equivalent to the hyperbolic metric. We will call this the flat-cylinder metric.

We also take this opportunity to introduce the \textit{transverse horospherical flow}. With $X$ as above, consider a horoball $B\subset X$ with horospherical boundary $H$ centered at the fixed point $\xi$ of a parabolic isometry $\psi$. The subgroup generated by $\psi$ preserves $H$ and $B$. The data $(B,H,\xi)$ determines a flow $\varphi_t: B\times [0,\infty)\to B$ defined by $$\varphi_t(p) = \alpha_{p,\xi}(t),$$ where $\alpha_{p,\xi}:[0,\infty)\to X$ is the unique geodesic starting from $p$ and tending towards $\xi$ at $\infty$.

 \begin{lem}
    The transverse horospherical flow is $\langle \psi\rangle$-equivariant.
 \end{lem}

\begin{proof}
 Notice $$\alpha_{\psi\cdot p, \xi}(0)=\psi \cdot p = \psi\cdot \alpha_{p,\xi}(0).$$ Since $\alpha_{\psi\cdot p, \xi}(t)$ and $\psi\cdot \alpha_{p,\xi}(t)$ describe geodesics with the same starting point and end point, they are identical. 
\end{proof} 

\begin{proof}[Proof of proposition \ref{tame}]
 By conformal invariance of energy we're permitted to do all of our computations in the flat-cylinder metric. Firstly let us assume there is a peripheral $\gamma$ such that $\rho(\gamma)$ is hyperbolic. Take any equivariant map $f:\tilde{M}\to X$ and fix a cusp neighbourhood associated to the peripheral and isometric to $U(\tau)$. As $\rho(\gamma)$ is hyperbolic, $$d_{g}(f(iy),f(\tau+iy))=d_{g}(f(iy),\rho(\gamma)f(iy))\geq \ell(\rho(\gamma))>0,$$ independent of $y$. For each $y$ let $\gamma_y$ be the path $x\mapsto f(x+iy)$, $x\in [0,\tau]$. The inequality above implies $$\ell(\rho(\gamma))\leq \int_0^\tau ||d\gamma_y||_{\sigma'}dy$$ and by Cauchy-Schwarz we obtain $$\frac{\ell(\rho(\gamma))^2}{2\tau}\leq \int_0^\tau \frac{1}{2}||d\gamma_y||_{\sigma'}^2dy\leq \int_0^\tau e(f)(x,y)dy.$$
  Hence,
    $$E(f)\geq E_{V}(f) = \int_a^\infty\int_a^\tau e(f)(x,y)dxdy\geq \frac{\ell(\rho(\gamma))^2}{2\tau}\int_a^\infty dy=\infty,$$
which shows all equivariant maps have infinite energy.

 For the other direction, we simply produce an equivariant finite energy map. We build a finite energy map in a neighbourhood of each cusp, equivariant with respect to the subgroup generated by $\rho(\gamma_j)$ and then extend smoothly to a $\rho$-equivariant map on the (compact) complement of the cusps.

 By induction it suffices to assume that there is only one cusp neighbourhood $V$. We identify it with some $U(\tau)$. Let $\gamma$ be the corresponding curve. If $\rho(\gamma)$ is elliptic then we simply map all of $V$ to a fixed point of $\rho(\gamma)$. This is clearly equivariant and has zero energy in $V$. Henceforward we assume $\rho(\gamma)$ is parabolic. $\langle \rho(\gamma) \rangle$ stabilizes a horoball $B$ with horopsherical boundary $H$. Let $g$ be any $C^\infty$ $\rho|_{\langle \gamma \rangle}$-equivariant map $\mathbb{R}\to H$. Define $f:\tilde{V}\to B$ by $$f(x+iy) = \varphi_{v \log (y+1)}(g(x)),$$ where $\varphi$ is the transverse horospherical flow with respect to the fixed point and $v>0$ will be specified later. We compute $$|df(\partial/\partial y)|_{f(x+iy)}= |\partial/\partial y (v \log (y+1))| = \frac{v}{y+1}.$$ Next, note that $$J_x(y) := \frac{\partial}{\partial x} f(x+iy)$$ is a Jacobi field for each $x$. By the curvature assumption on $X$, the Rauch comparison theorem shows that any Jacobi field on $X$ along a geodesic decays exponentially in time. By compactness of $S^1$ there is a $u>0$ such that $$|J_x(y)|\leq Ae^{-u\cdot v\log (y+1)}$$ for all $x$. Now choose $v$ so that $uv\geq 1$. Then $$|df(\partial/\partial x)|_{f(x+iy)}\leq \frac{A}{(y+1)^{uv}},$$ and furthermore $$E_{V}(f)\leq \int_0^{\infty}\int_0^\tau \frac{v^2+A^2}{2(y+1)^2}dxdy = \frac{\tau(v^2 + A^2)}{2} < \infty,$$ and the result follows. 
\end{proof}

\begin{remark}
The total energy of a harmonic map is finite if and only if the Hopf differential is integrable. Passing to polar coordinates, we see that an integrable holomorphic quadratic differential has a pole of order at most $1$ at a puncture.
\end{remark}

Suppose a representation admits a finite energy equivariant map. If it does not fix a point on the ideal boundary, the harmonic map determined by Theorem \ref{classic} is unique. If $\rho$ stabilizes a geodesic, there is a $1$-parameter family of harmonic maps that differ by translations along that geodesic axis. The standard methods push through to give a uniqueness criterion in our setting.

\begin{lem} \label{unique}
Let $M$ be a complete finite volume hyperbolic surface, let $(X,g)$ be $\textrm{CAT}(-1)$, and let $f_1$ and $f_2$ be equivariant harmonic maps for $\rho$ such that the map $z\mapsto d(f_1,f_2)(z)$ is bounded. If $\rho$ does not fix a point on $\partial_\infty X$ then $f_1=f_2$. If $\rho$ stabilizes a geodesic, then $f_1$ and $f_2$ may differ by translation along a geodesic.
\end{lem}
 
 \begin{proof}
For $z\in M$ let $\{e_1,e_2\}$ be an orthonormal frame for the tangent bundle in a neighbourhood of $z$ and let $\{v_1^0,\dots, v_n^0\}$, $\{v_1^1,\dots, v_n^1\}$ be orthonormal frames for neighbourhoods of $f_1(z)$, $f_2(z)$ respectively. In these frames we write $$(f_k)_*e_i = \sum_{m=1}^n \lambda_{i,m}^kv_m^k.$$ $\{v_1^0,\dots, v_n^0,v_1^1,\dots, v_n^1\}$ is an orthonormal frame near $(f_1(z),f_2(z))\in X\times X$. Define vector fields $X_i\in \Gamma(T(X\times X))$ so that around $(f_1(z),f_2(z))$ the projections onto the first and second factors are $f_1^*e_i$ and $f_2^*e_i$ respectively. Let $d:\tilde{M} \to \mathbb{R}$ be the function $$d(z)=d_{g\oplus g}(f_1(z),f_2(z))$$ which is $C^\infty$ away from the diagonal. From a computation in \cite[Chapter 11.2]{schoenyau}, if we assume $f_1(z)\neq f_2(z)$ then from the fact that the $f_k$ are harmonic, $$\Delta d^2\geq 2d\sum_{i=1}^2 D^2 d_g(X_i,X_i)$$ around $z$. Above, $\Delta$ is the Laplacian on $\tilde{M}$ and $D^2d_g$ is the Hessian of $d_g$, the distance function on $(X,g)$.

By equivariance, $d$ descends to a bounded subharmonic function on $M$.
As $M$ is parabolic in the potential theoretic sense, this function is constant. Therefore, $$2d\sum_{i=1}^2 D^2 d_g(X_i,X_i)=0.$$ This forces $d=0$  or $D^2d_g(X_i,X_i)=0$. In the first case we have $f_1=f_2$ so let us move to the latter. From an argument in \cite[Chapter 11.2]{schoenyau}, this implies either $f_1=f_2$ or $f_1$ and $f_2$ have image in a geodesic and differ by a translation along that geodesic. By equivariance, this last case can only occur if $\rho$ stabilizes a geodesic.
 \end{proof}
 
\end{subsection}

\begin{subsection}{An energy estimate}
The goal of this subsection is to prove Proposition \ref{energy}. Let $\sigma$ denote the hyperbolic metric on $\mathbb{H}$ with constant curvature $-1$. Unless otherwise specified, for the rest of the paper this is the metric on $\mathbb{H}$. We recall from \cite{Wan} a fundamental result in the theory of harmonic maps between surfaces.

\begin{thm}
(Wan) For any holomorphic quadratic differential $\Phi$ on $\mathbb{H}$, there is a (possibly non-surjective) harmonic diffeomorphism $h:\mathbb{H}\to \mathbb{H}$ such that 
\begin{itemize}
    \item $H(h)\geq 1$,
    \item The metric $H(h)\sigma(z)|dz|^2$ is complete on $\mathbb{H}$, and
    \item $\textrm{Hopf}(h)=\Phi$.
\end{itemize}
The harmonic map is unique up to isometries.
\end{thm}

We now have the machinery to state the energy estimate.

\begin{prop} \label{energy}
 Suppose $f$ is a harmonic map from $\mathbb{H}$ to a $\textrm{CAT}(-1)$ Hadamard manifold $(X,g)$. The energy density is always bounded above by that of any harmonic diffeomorphism $h:\mathbb{H}\to \mathbb{H}$ with the same Hopf differential and $H(h)\geq 1$. The inequality is strict unless $f$ takes $\mathbb{H}$ into a totally geodesic plane of constant sectional curvature $-1$.
\end{prop}

This is essentially a non-compact and generalized version of \cite[Lemma 2.1]{DT}. Proposition \ref{energy} is a consequence of the next lemma.
 
 \begin{lem} \label{betterenergy}
   $L(f)\leq H(h)$ and $H(f)\leq H(h)$ everywhere on $\mathbb{H}$, with equality for one of them if and only if and $f$  maps $\mathbb{H}$ diffeomorphically into a totally geodesic plane $\mathbb{H}\subset X$ of constant curvature $-1$. If equality holds at one point, it holds everywhere.
 \end{lem}
 
 \begin{proof} [Proof of proposition \ref{energy}]
 
 Indeed, assuming the above result, if $f$ is not such an embedding with totally geodesic image, then $L(f)<H(h)$ and $H(f)<H(h)$ everywhere. As $$H(f)L(f)=||\Phi||^2 = H(h) L(h),$$ we obtain $L(f)>L(h)$, $H(f)>L(h)$. If $H(f)-L(f)\geq 0$, then $(H(f)-L(f))^2<(H(h)-L(h))^2$. If $H(f)-L(f)\leq 0$, then $L(f)> L(h)$ and $L(h)-H(h)<L(f)-H(f)$, so  $(H(f)-L(f))^2<(H(h)-L(h))^2$ holds here as well. By adding $4H(f)L(f)=4H(h)L(h)$ to both sides we have $$e(f)^2 = (H(f)+L(f))^2 < (H(h)+L(h))^2 = e(h)^2,$$ which yields the desired result. 
 \end{proof}
 
 We will prove Lemma \ref{betterenergy} via the Bochner formula. Our main tool is the generalized maximum principle of Omori-Yau (see \cite{omori} for the version we use and also \cite[Theorem 3]{chengyau} for the extension to manifolds with a lower bound on the Ricci curvature).
 
 \begin{lem} 
 (Omori) Let $M$ be a Riemannian manifold such that all sectional curvatures are bounded from below. Let $f$ be a $C^2$ function on $M$ that is bounded above. There is a sequence $(x_n)_{n=1}^\infty$ such that $f(x_n)\to \sup f$ and $$|\nabla f(x_n)|\to 0 \hspace{1mm} ,  \hspace{1mm} \limsup_{n\to \infty}\Delta f(x_n)\leq 0$$ as $n\to \infty$. 
 \end{lem}

\begin{proof}[Proof of lemma \ref{betterenergy}]
   By \cite[Corollary 3]{sampson}, the set $\mathcal{D}$ on which $f^*g$ is non-degenerate is either empty or open and dense. From \cite[page 10]{schoenyau} either $H(f)=0$ everywhere or the zeroes are isolated. We replace $\mathcal{D}$ with the open dense set $$U:=\mathcal{D}-\{z:H(f)(z)=0\}.$$ 
It follows from the Bochner formula that 
\begin{itemize}
    \item $H(h)\geq 1$ and solves the PDE $$\Delta \log H(h) = 2H(h) -2\frac{||\Phi||^2}{H(h)}-2.$$
    \item On $U$, $H(f)$ solves the PDE $$\Delta \log H(f) = -2\kappa(f^*g) H(f) + 2\kappa(f^*g)\frac{||\Phi||^2}{H(f)}-2.$$ 
\end{itemize}

The next result is essentially \cite[Theorem 4]{sampson}. It was tweaked to its present form in \cite{DT}.
 
 \begin{lem} \label{curve}
For all $x\in U$, $\kappa(f^*g)\leq -1$. We have equality iff the second fundamental form of $f(\mathbb{H})$ vanishes at $x$. In particullar, $\kappa(f^*g)=-1$ everywhere on $U$ iff $f(U)\subset X$ is totally geodesic.
 \end{lem}
 
 When the metric $f^*g$ is degenerate or if $H(f)\leq L(f)$, $$H(f)\leq L(f)=||\Phi|| = H(h)^{1/2}L(h)^{1/2}<H(h).$$ Hence we can dismiss the case $U=\emptyset$ and furthermore we're allowed to work only on $U$. As stated previously $H(h)\geq 1$ everywhere, so that $H(f)/H(h)$ never vanishes on $U$. Assume for the sake of contradiction that $H(f)>H(h)$ at a point $x$. Necessarily, $H(f)(x)\geq L(f)(x)$. From the Bochner formula we have 
 \begin{align*}
     \Delta \log(H(f)/H(h)) &= 2(H(h)-H(f)) + 2||\Phi||^2(H(h)^{-1}-H(f)^{-1}) \\
     &- 2(\kappa(f^*g)+1)(H(f)-L(f)).
 \end{align*}
  Let $w:=\log (H(f)/H(h))$. Since $\kappa(f^*g)\leq -1$ and $H(h)\geq L(h)$, one can simplify the above equation to $$\Delta w \geq 2(H(f)-H(h))(1+||\Phi||^2/(H(h)H(f)))= 2(H(h)+L(h))(e^w -1),$$ and hence $$\Delta w \geq 2(e^w -1)>0$$ at such an $x$. It now follows that this point $x$ cannot be a local maximum for $H(f)/H(h)$ as otherwise $$0\geq \Delta w >0.$$ Thus, there is a sequence contained in $U$ and tending to the boundary of $\mathbb{H}$ along which $H(f)/H(h)>1$ and  increases to $\sup H(f)/H(h)$. We argue this supremum is finite. Let $b>a>0$ and define $F: [b,\infty)\to \mathbb{R}$ by $$F(s)=\int_b^s \Big ( \int_a^t e^{2\tau} d\tau \Big )^{-1/2}dt +1.$$ $F$ is monotonically increasing, bounded above, and $F''< 0$ everywhere. Extend $F$ smoothly to $\mathbb{R}$ so that it is still monotonic and satisfies $$\lim_{t\to -\infty} F(t) > 0.$$ For some large $N>0$, let $\eta:\mathbb{H}\to [0,1]$ be a $C^\infty$ function that is $0$ on the open set $\{z: e^{w(z)}<1/2N\}$ and is $1$ on $\{z:  e^{w(z)}\geq 1/N\}$. Then $F\circ (\eta w)$ is a bounded $C^\infty$ function on $\mathbb{H}$. By the Omori-Yau maximum principle there is a sequence $(x_n)_{n=1}^\infty$ escaping to the boundary such that $$F\circ w(x_n)=F\circ (\eta w)(x_n)\to \sup F\circ \eta w=\sup F\circ w$$ as $n\to \infty$ and $$|\nabla F\circ w(x_n)|=|\nabla F\circ(\eta w)(x_n)|\to 0 \hspace{1mm} , \hspace{1mm} \limsup_{n\to \infty} \Delta F\circ w(x_n)=\limsup_{n\to \infty} \Delta F\circ (\eta w)(x_n)\leq 0.$$ Above, we removed finitely many points in the sequence so we can assume $w(x_n)> 2/N$ always, and we also used the fact that $w=\eta w$ in this region.
  
An analogue of the computation below is contained in \cite[Section 5]{chengyau}, where they work with global subsolutions. We may choose a subsequence of the $x_n$, and abuse notation by still labelling it $x_n$, so that $$0\leq \frac{F'(w)|\nabla w|}{F(w)^{2}}(x_n)\leq 1/n \hspace{1mm} (*)$$ and $$-\frac{F''(w)|\nabla w|^2}{F^2}(x_n) -\frac{F'(w)\Delta w}{F^2}(x_n)+\frac{(F')^2|\nabla w|^2}{F^3}(x_n)\geq - 1/n.$$ Multiplying the above by $(F'(w))^2/F(w)^2|F''(w)|$ we obtain
\begin{align*}
    &-\frac{F''(w)}{|F''(w)|}\frac{F'(w)^2|\nabla w|^2}{F^4}- \frac{F'(w)^3}{F^4}\frac{\Delta w}{|F''(w)|}+ \frac{F'(w)^2}{F(w)|F''(w)|}\frac{F'(w)^2|\nabla w|^2}{F(w)^4}\\
    &\geq\frac{-F'(w)^2}{nF(w)^2|F''(w)|}
\end{align*}
at $x_n$. Note $F$ satisfies $$\limsup_{s\to \infty} \frac{|F'(s)|^2}{F(s)|F''(s)|}<\infty,$$ and combining this with the line above yields that $$\frac{1}{n^2}- \frac{F'(w)^3}{F^4(w)}\frac{\Delta u}{|F''(w)|}+\frac{A}{n^2}\geq -\frac{A}{n}.$$ Using $\Delta w \geq 2(e^w-1)$ at $x_n$ we infer $$\frac{F'(w)^3(e^{2w}-1)}{F^4(w)|F''(w)|}(x_n) \lesssim \frac{1}{n}$$ as $n\to \infty$. However, it is straightforward to compute $$\liminf_{s\to \infty}\frac{F'(s)^3(e^{2s}-1)}{F(s)^4|F''(s)|}>0.$$ This means $\limsup_{n\to \infty}w(x_n)=\infty$ is impossible.

To understand this supremum we apply the generalized maximum principle to the function $\eta w$. As with $F\circ (\eta w)$, there is a sequence $(y_n)_{n=1}^\infty$ leaving all compact subsets of $\mathbb{H}$ such that after refining if necessary so that $w(y_n)> 2/N$, $$w(y_n)=\eta w(y_n)\to \sup \eta w=\sup w$$ and $$0\geq \limsup_{n\to \infty} \Delta \eta w(y_n) = \limsup_{n\to \infty} \Delta w(y_n) \geq 2(e^{w(y_n)}-1)\geq 0.$$
This forces $\sup H(f)/H(h)= 1$, which contradicts our assumption that $H(f)>H(h)$ at least once. Hence, $H(f)\leq H(h)$ always. If $L(f)>H(h)$ at a point, we can precompose $f$ with the complex conjugation to find a harmonic map $\overline{f}$ with $H(\overline{f})>H(h)$, so this is also impossible. Now that we have our inequalities, a special case of \cite[Theorem 1]{minda} indicates when this inequality is strict.

\begin{lem}
Let $u$ be a real non-positive function on a domain $V$ in the complex plane such that $\Delta u\geq Au$ for a constant $A>0$. Then either $u=0$ on $V$ or $u<0$ on all of $V$.
\end{lem}

With this in mind, take an increasing exhaustion $(D_k)_{k=1}^\infty$ of $\mathbb{H}$ by pre-compact open sets. $e^x \geq x+1$ gives $$\Delta w \geq (H(f)-H(h))w\geq [\max_{D_k}(H(f)-H(h))]w$$ in $D_k\cap U$. It follows that either $H(f)=H(h)$ or $H(f)<H(h)$ everywhere. If $H(h)=H(f)$ then $L(h)=L(f)$ and we see that $f^*g$ is non-degenerate everywhere. By the Bochner formula above this forces $\kappa(f^*g)=-1$, and so by Lemma \ref{curve} $f$ maps $\mathbb{H}$ diffeomorphically into a totally geodesic plane. Identifying this plane with $\mathbb{H}$, the formulas $$h^*\sigma=e(h)\sigma+\Phi + \overline{\Phi} \hspace{3mm} , \hspace{3mm} f^*g = e(f)\sigma + \Phi + \overline{\Phi}$$ show that $f$ differs from $h$ by an isometry of $\mathbb{H}$. This implies $f: (\mathbb{H},h^*\sigma)\to (X,g)$ is an isometric embedding. If $L(f)=H(h)$ at a point, we precompose with the complex conjugation, producing a harmonic map $\overline{f}$ with $H(\overline{f})=H(h)$ at a point, so that by the above it takes $\mathbb{H}$ isometrically onto a totally geodesic plane of constant curvature $-1$. Undoing the conjugation, we recover the same result for $f$.
\end{proof}
\end{subsection}
\end{section}

\begin{section}{Fricke-Teichm{\"u}ller Spaces}
\begin{subsection}{Harmonic diffeomorphisms}
As before, $M=\tilde{M}/\Gamma$ is a complete finite volume hyperbolic surface. Label the punctures of $M$ by $p_1,\dots, p_n$.

\begin{defn}
The \textit{Fricke-Teichm{\"u}ller space} is the subset of the character variety $\textrm{Hom}(\Gamma,\textrm{PSL}_2(\mathbb{R}))/\Gamma$ consisting of conjugacy classes of geometrically finite representations.
\end{defn}

Each representation in this space is the holonomy of a geometrically finite hyperbolic structure on $M$. Fix an $n$-tuple $(\ell_1,\dots,\ell_n)\in \mathbb{R}_{\geq 0}^n$. 

\begin{defn}
Let $T(M,\ell_1,\dots, \ell_n)$ be the subspace of the Fricke-Teichm{\"u}ller space such that convex core of the underlying surface associated to each representation has, for each $\ell_k$, either a puncture if $\ell_k=0$ or a closed geodesic boundary component of length $\ell_k\neq 0$.
\end{defn}

 When the context is clear we just call this the Teichm{\"u}ller space. We represent points as equivalence classes $[S,f]$, where $S$ is a surface and $f$ is a diffeomorphism $f:M\to C(S)$. Another point $[S',f']$ is equivalent if $f^{-1}\circ f'$ is an isometry and isotopic to the identity.
 
 We uniformize to obtain a compatible holomorphic structure on $M$. Around any puncture $p$ we choose a local coordinate $z$ such that $z(p)=0$. A meromorphic quadratic differential $\Phi$ with a pole of order $2$ at such a puncture admits a Laurent expansion $$(a_{-2}z^{-2}+a_{-1}z^{-1}+a_0+\dots)dz^2.$$ The $a_{-2}$ term is invariant under holomorphic coordinate changes that take $0$ to $0$, and correspondingly we call it the \textit{residue} of $\Phi$ at $p$. 
 
  For ease of notation we assume $\ell_{1},\dots, \ell_{d_1}=0$, $\ell_{d_1+1},\dots, \ell_{d_2}>0$, $d_1+d_2=n$. For any $(d_2-d_1)$-tuple of unit norm complex numbers $\theta_{d_1+1},\dots, \theta_{d_2}$ let $P$ be the vector $(\ell_k,\theta_k)_{k=d_1+1}^n$.

\begin{defn}
$Q(M,P)$ is the space of meromorphic quadratic differentials on $M$ with poles of order at most $2$ at the $p_k$ and residues $$-\Lambda(\theta_k)\ell_k^2/16\pi^2.$$ If $\ell_k=0$ we have a pole of order at most $1$.
\end{defn}

The space of holomorphic quadratic differentials $Q(M)$ with pole-type singularities at the cusps of $M$ is a Fr{\'e}chet space with seminorms coming from the restriction of the $L^1$ norm to pre-compact open sets. $Q(M,P)$ inherits the subspace topology from $Q(M)$. The following result can be deduced from the work of Wolf in \cite{Wolf}. It links the two spaces above.

\begin{thm} (Wolf) For any $[S,f]\in T(M,\ell_1,\dots, \ell_n)$ there is a unique harmonic diffeomorphism $$h_f:M\to C(S)$$ in the isotopy class such that $\textrm{Hopf}(h_f)\in Q(M,P)$.
\end{thm}

This allows us to define a map $$\Psi: T(M,\ell_1,\dots, \ell_n)\to Q(M,P)$$ by $[S,f]\mapsto \textrm{Hopf}(h_f)$.

\begin{remark}
In \cite{Wolf}, Wolf only explicitly computes and writes down the residue in the event $\theta=0$, although he outlines constructions for $\theta\neq 0$. The values listed above can be computed by following the proof of Proposition \ref{compute} in the current paper.
\end{remark}
\end{subsection}
\begin{subsection}{Proof of Theorem 1.4}
The content of Theorem 1.4 is that the map $\Psi$ is a diffeomorphism. The first step is a dimension count.

\begin{lem} \label{dim}
$T(M,\ell_1,\dots, \ell_n)$ and $Q(M,\ell_1,\dots, \ell_n)$ are homeomorphic to $\mathbb{R}^{6g-6+2n}$.
\end{lem}

\begin{proof}
 For the Teichm{\"u}ller space, view the punctures as nodes and double the surface across the boundary. By mapping an element to this double, $T(M,\ell_1,\dots, \ell_n)$ then embeds into a strata of the augmented Teichm{\"u}ller space consisting of surfaces of genus $2g+d_2-1$ with $d_1$ nodes. Choosing a pants decomposition that includes all of our boundary curves and nodes (the nodes correspond to pinched curves) and taking the corresponding Fenchel-Nielsen coordinates shows this strata has dimension $$12g+6d_2+4d_1-12.$$ Every curve in the image of $T(M,\ell_1,\dots, \ell_n)$ has an involutive symmetry across the boundary, and so the image is determined by at most $6g-6 +3d_2+2d_1$ coordinates. Fixing the lengths of the boundary curves kills another $d_2$ parameters and we obtain $6g-6+2n$.
 On the other hand, the space of holomorphic quadratic differentials with poles of order bounded above by $k_1,\dots, k_n$ at $p_1,\dots, p_n$ forms a vector space over $\mathbb{C}$ and by Riemann-Roch it has real dimension $$6g-6+2\sum_j k_j.$$ Specifying the Laurent expansion at the poles then removes $2$ parameters for each puncture and we end up with $6g-6+2n$ degrees of freedom.
\end{proof}

For a closed arc $c$ on a hyperbolic surface, let $\ell(c)$ denote the hyperbolic length of the geodesic representative. Below, the surface on which the curve lives will be clear.

\begin{proof}[Proof of Theorem \ref{frick}]
By Brouwer's invariance of domain, it is enough to show $\Psi$ is continuous, injective, and proper. Continuity and injectivity follow from arguments in \cite[Section 4]{Wolf}, so we only need properness. To this end, let $K\subset Q(M,P)$ be compact. Remove cusp neighbourhoods around all punctures, each one chosen small enough so that all simple closed geodesics of $M$ are contained in the resulting subsurface, which we will call $M'$. By a estimate from \cite[Lemma 3.2]{wolf1} $$E_{M'}(h_f)\leq 2\int_{M'} |\Phi|+\textrm{Area}(h_f(M'))\leq 2\int_{M'} |\Phi|+\textrm{Area}(h_f(M)).$$ The Gauss-Bonnet theorem yields $$E_{M'}(h_f)\leq 2\int_{M'} |\Phi|-2\pi\chi(M).$$ 
By a minor and well-understood modification of the proof of the Courant-Lebesque lemma \cite[Lemma 3.1]{cllemma}, we obtain $\ell_{Y'}(h_f'(\gamma))\leq A_{\mathcal{F}}$ for any finite collection $\mathcal{F}$ of simple closed geodesics inside $S$ and any choice of representative pair $(Y',h_f')\in [(Y',h_f')]\in \psi^{-1}(K)$. Since the boundary lengths are fixed we have an upper bound on the lengths of finite collections of simple closed geodesics in all of any $Y'$. We argue that we also have a uniform lower bound on such lengths. On a complete finite volume hyperbolic surface, any essential simple closed geodesic $\delta$ is contained in an embedded annulus. This annulus has a horizontal coordinate specified by $\delta$ and an orthogonal vertical coordinate. Any simple closed geodesic $\delta'$ that transversely intersects $\delta$ once must pass through the entire vertical length of the annulus. If we have a geodesic $\delta$ such that $\ell(h_{f_k}(\delta))$ shrinks to $0$ along some sequence $(Y_k,h_{f_k})$, select a curve $\delta'$ in $M$ as above. From the collar lemma we see that $\ell(h_{f_k}(\delta'))\to \infty$ as $k \to \infty$. However, we can uniformly bound $\ell(h_{f_k}(\delta'))$ from above, so this is impossible.

Now, view the punctures on $M$ as nodes and double across all punctures that ``are opened" to get a noded surface $M^d$. Likewise double all surfaces $(Y,h_f)\in \Psi^{-1}(K)$ across the boundaries. $h_f$ extends by reflection and we get a pair $(Y^d,h_f^d)$. This provides a map $$\iota: T(M)\to T(S_{2g+d_1-1,2d_2})$$ that is a diffeomorphism onto its image. By \cite[Lemma 3.3]{ursula} on any $S_{g,n}$ there is a collection of simple closed curves $\delta_1,\dots, \delta_{6g-5+2n}$ so that the map $$\mathcal{L}_{g,n}: T(S_{g,n})\to \mathbb{R}^{6g-5+2n}$$ given by $$[X,\phi]:=\chi \mapsto (\ell_\chi(\delta_1),\dots, \ell_\chi(\delta_{6g-5+2n}))$$ is a diffeomorphism onto its image.  The composition $\mathcal{L}_{2g+d_1-1,2d_2}\circ \iota$ takes $\Psi^{-1}(K)$ into a compact set, and hence $\Psi$ is proper. As discussed above, this completes the proof.
\end{proof}
\end{subsection}
\end{section}
\begin{section}{Equivariant harmonic maps with poles of order $2$}
We prove Theorem \ref{first} in a number of steps.
\begin{subsection}{Existence of harmonic maps}
Once and for all, fix a complete finite volume hyperbolic surface $M=\tilde{M}/\Gamma$, a $\textrm{CAT}(0)$ Hadamard manifold $(X,g)$, and a reductive representation $\rho: \Gamma \to \textrm{Isom}(X,g)$. We denote both the metric on $M$ and its lift to $\tilde{M}$ by $\sigma$.

There is a finite set of punctures $p_1,\dots, p_n$ with associated peripheral isometries $\gamma_1,\dots, \gamma_n$ such that $\rho(\gamma_j)$ is hyperbolic. If this set is empty then $\rho$ admits a finite energy equivariant map, for which the existence is already known. Hence we declare $n\geq 1$ and by an induction argument we may reduce to $n=1$. Let us now fix some notation: set $\gamma:=\gamma_1$ and write $M=M^c\cup C$, $C$ is a cusp neighbourhood corresponding to $\gamma$ and $M^c$ is the complement. $C$ is isometric to $U(\tau)$ for some $\tau>0$ and we equip it with the relevant coordinates $x+iy$, $0\leq x\leq \tau$, $y>a$. For $r>a$, $C_r$ will be $\{x+iy\in C: a<y\leq r\}$ and $M_r=M^c\cup C_r$. Let $D$ be some fixed fundamental domain with respect to the covering $\pi: \tilde{M}\to M$ and analogously we put $D^c=D\cap \pi^{-1}(M^c)$, $D'=D\cap \pi^{-1}(C)$, $D_r=D\cap \pi^{-1}(C_r)\cup D^c$. We set $i_r:M_r\to M$ to be the inclusion map. Finally, we use $\ell(\cdot)$ to denote the length of a rectifiable curve on a surface and hope there is no confusion with isometries.

\begin{prop} \label{existence}
  Given the data $M,X,\rho$ as above, there exists a $\rho$-equivariant harmonic map $f:\tilde{M}\to X$.
\end{prop}

\begin{proof}
  
 Let $\alpha:[0,\tau]\to X$ be a constant speed curve with image in the axis of $\rho(\gamma)$ and so that $\alpha(\tau)=\rho(\gamma)\alpha(0)$. By \cite{Corlette2} there exists a unique harmonic section $s_r$ of the pullback bundle $i_r^*X\to M_r$ with boundary values $\alpha$. Extend $s_r$ to $M$ via $s_r(x,t)=s_r(x)$. The $s_r$ induce equivariant maps $f_r:\tilde{M}\to X$, that are harmonic on $\pi^{-1}(M_j)$. We prove the $f_r$ converge along a subsequence in the $C^\infty$ topology to an equivariant harmonic map.

Let $\varphi$ be any non-harmonic equivariant map corresponding to a section of $i_{0}^*X\to M^c$ with boundary values $\alpha$.  As with $f_r$, define $\varphi$ on the rest of $D$ by $\varphi(x,t)=\varphi(x)$ and then extend equivariantly to $\tilde{M}$. Let $\beta$ be the image of $\alpha$ on the geodesic axis of $\rho(\gamma)$ and set $$\beta_t^r:=f_r([0,\tau]\times \{t\}).$$  Notice that $|d\varphi|_{\sigma'}=\ell(\beta)/\tau$ on $C$ since it has constant speed. For $r>s$,  $$2E_{C_r\backslash C_s}(\varphi) = \int_{C_r\backslash C_s} |d\varphi|_{\sigma'}^2 dv_{\sigma'}= (r-s)\ell(\beta)^2/\tau.$$ As $\beta$ is a geodesic arc in a negatively curved space, $$s\ell(\beta)\leq \int_0^s\ell(\beta_t^r) dt,$$ and hence for any $r>s$, $$E_{C_r\backslash C_s}(\varphi) \leq \frac{1}{2}\int_s^r\ell(\beta_t^r)^2/\tau dt\leq\frac{1}{2}\int_s^r\Big (\int_{S^1\times \{t\}} |df_r|d\theta \Big )^2\tau^{-1}dt\leq E_{C_r\backslash C_s}(f_r).$$
From the non-positive curvature hypothesis $f_r$ minimizes energy among maps to $X$ with the same equivariant boundary values. In particular, $$E_{M_r}(f_r)\leq E_{M_r}(\varphi),$$
and moreover $$E_{M_s}(f_r) = E_{M_r}(f_r)-E_{M_r\backslash M_s}(f_r)\leq E_{M_r}(\varphi)-E_{M_r\backslash M_s}(\varphi) =E_{M_s}(\varphi).$$
By a classical PDE estimate (for a source see \cite[Lemma 5.3.4]{heat}), $$\sup_{D_s} e(f_r) = \sup_{M_s}e(f_r)\leq A_s E_{M_{s+1}}(f_r)\leq A_sE_{M_{s+1}}(\varphi),$$ where $A_s$ depends on the Ricci curvature of $M_{s+1}$, the injectivity radius on $M_s$, and $\textrm{dist}(\partial M_s,\partial M_{s+1})$. Since $\rho$ is acting by isometries we get the same bound in all of $\pi^{-1}(M_r)$. Next, we claim there is a compact set $O_s\subset X$ such that $$f_r(D_s)\subset O_s$$ for all $r$. Appealing to the energy density bound above, it is enough to show that for a fixed point $x_0\in D_s$, $f_r(x_0)$ stays within some compact set as $r\to \infty$. We find it convenient from here to split cases. Firstly, let us assume that the image of $\rho$ does not lie in a parabolic subgroup. Let $\xi$ be a point in the boundary at infinity $\partial_\infty X$. There is loop $\gamma: [0,L]\to M$ parametrized by arclength such that $$\rho(\gamma)(\xi) \neq \xi.$$ Choose $\ell$ so that the image of $\gamma$ under $\pi$ lies entirely in $M_\ell$ and let $A_\ell$ be a uniform bound on the derivative in $\pi^{-1}(M_\ell)$. We then have, for $r>\ell$, $$d_g(\rho(\gamma)f_r(x_0),f_r(x_0))=d_g(f_r(\gamma (x_0)),f_r(x_0))\leq A_\ell L.$$ This is because lifting $\gamma$ to the universal cover gives a path between $x_0$ and $\gamma\cdot x_0$ that remains within lifts of $M_\ell$. Choose a neighbourhood $B$ of $\xi$ in $X\cup\partial_\infty X$ such that $$d_g(B\cap X,\rho(\gamma)B\cap X)> A_\ell L.$$ Then $f_r(x_0)$ cannot enter $B$, no matter how large $r$ grows. Via compactness we find a finite number of neighbourhoods $(B_k)_k$ as above that cover the boundary sphere. Choosing $O_s:=X\backslash (\cup_k N_k)$ the claim follows. Notice then that $f_r$ takes any lift of $M_s$ to a compact set: $$f_r(\gamma D_s)\subset \rho(\gamma)O_s$$ for all $r,s$. It now follows by a well-known argument, namely an application of the Arzel{\`a}-Ascoli theorem and a bootstrap, that a subsequence of the $(f_r)_{r>0}$ converges uniformly on compact subsets of $\tilde{M}$ to a harmonic map $f_\infty$. By equivariance of the $f_r$ on $\pi^{-1}(M_r)$, $f_\infty$ is necessarily equivariant. 

We next treat the case where $\rho$ stabilizes a totally geodesic flat $F$. $F$ is a symmetric space and identifies isometrically as $$G/H:=(O(n)\rtimes \mathbb{R}^n)/O(n).$$ Fix two points $x_0\in D_s$ and $y_0\in F$ and for each $r$ choose $g_r\in G$ such that $g_rf_r(x_0)=y_0$. We notice that for any $y\in F$ and $\gamma\in \Gamma$, $d(g_r\rho(\gamma)g_r^{-1}y,y)$ is uniformly bounded in $r$. Indeed,
\begin{align*}
    d_g(g_r\rho(\gamma)g_r^{-1}y,y) &\leq d_g(g_r\rho(\gamma)g_r^{-1}y,g_r\rho(\gamma)g_r^{-1}y_0) + d_g(g_r\rho(\gamma)g_r^{-1}y_0,y_0) + d_g(y,y_0) \\
    &= 2d_g(y,y_0) + d_g(g_r\rho(\gamma)g_r^{-1}y_0,y_0) \\
    &= 2d_g(y,y_0) + d_g(g_r\rho(\gamma)f_r(x_0), g_rf_r(x_0)) \\
    &= 2d_g(y,y_0) + d_g(g_rf_r(\gamma\cdot x_0), g_rf_r(x_0)) \\
    &= 2d_g(y,y_0) + d_g(f_r(\gamma\cdot x_0), f_r(x_0)),
\end{align*}
and we know $f_r$ has a uniform energy density bound on $\rho(\Gamma)\cdot D_r$. By the  argument of \cite[Lemma 2]{jostyau} there is a sequence $(r_n)_{n=1}^\infty$ increasing to $\infty$ and an element $g_\infty\in G$ such that for every $\gamma \in \Gamma$ and $y\in F$, $$\lim_{n\to \infty} g_{r_n}\rho(\gamma)g_{r_n}^{-1}y=g_\infty\rho(\gamma)g_\infty^{-1}y.$$  The orbit of the point $x_0$ under the family of maps $g_{r_n}f_{r_n}$ is a singleton, and by our uniform energy bound we see as above that there is a compact set $O_s$ such that $$g_{r_n}f_{r_n}(D_s)\subset O_s.$$  Arguing as above there is a subsequence along which $g_{r_n}f_{r_n}$ converges to a harmonic map $f_\infty$. Note that $g_{r_n}f_{r_n}$ is $g_{r_n}\rho(\Gamma)g_{r_n}^{-1}$-equivariant, so that $f_\infty$ is $g_\infty\rho(\Gamma)g_\infty^{-1}$-equivariant. Therefore, we may take $f:= g_\infty^{-1}f_\infty$ as the sought harmonic map. \end{proof}

 We use the ideas above to build a family of harmonic maps, indexed by a real parameter $\theta\in \mathbb{R}$. We perform a \textit{fractional Dehn twist} on each cylinder $C$. This is the map given in the cusp coordinates by $$x+iy\mapsto x+\theta y + iy$$ on $C$ and the identity map on the rest of $M$. Lift to a map $d^\theta$ on $\tilde{M}$. The lift commutes with the relevant parabolic isometry. Define $f_r^\theta$ to be the equivariant harmonic map on $\pi^{-1}(M_r)$ with the same equivariant boundary values as $\varphi\circ d^\theta|_{\partial D_r}$. Then extend to agree with $\varphi\circ d^\theta$ on the complement. The derivative matrix of $d^\theta$ is $$\begin{pmatrix} 1 & \theta \\ 0 & 1 \end{pmatrix},$$ so that $$||d(\varphi\circ d^\theta)||\leq ||d\varphi||(1+\theta).$$
Thus on $M_s$, $$E_{M_s}(f_r^\theta)\leq E_{M_s}(f_r)(1+\theta)^2.$$  By the argument of Proposition \ref{existence} there is a subsequence along which the $f_r^\theta$'s converge to a limiting harmonic map $f^\theta$. Of course, $f=f^0$.  

We keep the same characters $\alpha$ and $\beta$ from the proof of the above proposition. Note $\ell(\beta)= \ell(\rho(\gamma))$. Define $\varphi^\theta:=\varphi \circ d^\theta$. In local  Euclidean coordinates, $d^\theta$ is harmonic on $C$. Since $\nabla d\varphi=0$ on $C$, the composition is a harmonic map there (see \cite[Proposition 2.20]{ogtext}).

\begin{lem} \label{key}
The function $z\mapsto d(f^\theta,\varphi^\theta)(z)$ is uniformly bounded.
\end{lem}

\begin{proof}
   Let $\psi_r:=d(f_r^\theta,\varphi^\theta)$. By equivariance, each $\psi_r$ descends to a function on $M$. $\psi_r=0$ on $M\backslash M_r$, and since $\psi_r>0$ at some point we know it attains a maximum at a point in the interior of $M_r$.  As $\psi_r$ is subharmonic on $C_r$, $\sup_{z\in C_r}\psi_r(z)$ occurs on $\partial M^c$ and moreover $\psi_r$ is maximized at a point in $M^c$. Meanwhile, $$\psi_r\to d^2(f^\theta,\varphi^\theta)$$ uniformly on compacta as $r\to \infty$.
By smoothness, $\psi$ is uniformly bounded on $M^c$. This implies we have a uniform bound on the $\psi_r$'s inside $M^c$ as $r\to \infty$. Since the relevant maximum is attained inside $M^c$, this bound holds everywhere. 
\end{proof}

Let $\Phi:=\textrm{Hopf}(f^\theta)$. The context is clear so we do not include a $\theta$ in our notation. By equivariance we can view $\Phi$ as a holomorphic quadratic differential on any quotient of $\tilde{M}$ by a subgroup of $\Gamma$.

\begin{lem} \label{pole}
$\Phi$ has a pole of order $2$ at the cusp.
\end{lem}

\begin{proof}
  From the infinite energy phenomena, $\Phi$ either has a pole of order at least $2$ or an essential singularity. The $(2,0)$ component of the pullback metric by $\varphi^\theta$ is a section of $S^2(T^*M)$ that is holomorphic on $C$. We still denote it by $\textrm{Hopf}(\varphi^\theta)$. 
  
  We compute this differential in $C$. Choose a local orthonormal basis $\partial/\partial x$, $\partial/\partial y$ of the relevant tangent spaces so that $\partial\varphi^0/\partial y = 0$ always. Starting with $\theta=0$, we know that in local coordinates $$\textrm{Hopf}(\varphi^0)(z) = \frac{1}{4}\Big (|\partial \varphi^0/\partial x|^2 -|\partial \varphi^0/\partial y|^2 - 2i\langle \partial \varphi^0/\partial x,\partial \varphi^0/\partial y\rangle\Big ) dz^2.$$ Since $\varphi^0$ is constant in the vertical direction $$\textrm{Hopf}(\varphi^0)(z) = \frac{1}{4}|\partial \varphi^0/\partial x|^2dz^2=\ell(\rho(\gamma))^2/4\tau^2dz^2.$$ From the chain rule, $d\varphi^0$ and $d\varphi^\theta$ admit matrix representations with $$d\varphi^0 = \begin{pmatrix} v & 0 \end{pmatrix} \hspace{1mm} , \hspace{1mm} d\varphi^\theta = \begin{pmatrix} v & \theta v \end{pmatrix},$$ where $v$ is a $1\times \dim X$ column vector. Thus, 
\begin{align*}
    \textrm{Hopf}(\varphi^\theta)(z) &= \frac{1}{4}(|\partial \varphi^\theta/\partial x|^2 - |\partial \varphi^\theta/\partial y|^2 -2i\langle \partial \varphi^\theta/\partial x,\partial \varphi^\theta/\partial y\rangle)dz^2 \\
    &= \frac{1}{4}(1-\theta^2 - i2\theta)|\partial \varphi^0/\partial x|^2dz^2. 
    \end{align*} 
We take the strip conformally to a punctured disk via $$z\mapsto \zeta(z)=e^{\frac{2\pi i z}{\tau}},$$ taking the point at $\infty$ to $0$. The transformation law multiplies by $-\zeta^{-2}\tau^2/4\pi^2$, and we see that we have a pole of order $2$ with residue $$-\Lambda(\theta)\ell(\rho(\gamma))^2/16\pi^2.$$  We now compare $\Phi$ to $\textrm{Hopf}(\varphi^\theta)$. As $\varphi^\theta$ has rank $1$, the formula $J=H-L$ implies $$H(\varphi^\theta)^{1/2}=L(\varphi^\theta)^{1/2}=\frac{1}{\sqrt{2}}e(\varphi^\theta)^{1/2},$$ so that $\textrm{Hopf}(\varphi^\theta)=\sigma H(\varphi^\theta)^{1/2}L(\varphi^\theta)^{1/2}=\sigma e(\varphi^\theta)/2$. From Young's inequlaity, $$||\Phi||=\sigma H(f)^{1/2}L(f)^{1/2}\leq \frac{1}{2}\sigma e(f^\theta),$$ and hence it is enough to bound $e(f^\theta)$ by a sublinear function of $e(\varphi^\theta)$. This is not hard: for any $x_0\in \tilde{M}$, $r_0>0$, and $y \in B(x_0,r_0)$,
\begin{align*}
    d(f^\theta(x_0),f^\theta(y)) &\leq d(f^\theta(x_0),\varphi^\theta(x_0)) + d(f^\theta(y),\varphi^\theta(y))+ d(\varphi^\theta(x_0),\varphi^\theta(y)) \\
    &\leq A+ \sup_{B(x_0,r_0)} ||d\varphi^\theta||d(x_0,y).
\end{align*}
 Working in the flat cylinder metric, Cheng's lemma then gives $$||df||(x_0)\lesssim \frac{1+r_0}{r_0}(1+\sup ||d\varphi^\theta|| r_0).$$ 
 In a cusp neighbourhood, the injectivity radius of the flat cylinder metric is uniformly bounded below, and hence we may choose $r_0$ uniformly bounded below. Squaring for the energy density gives the desired bound.
\end{proof}

By equivariance, $f^\theta$ and $\varphi^\theta$ induce quotient maps $$f_\gamma, \varphi_\gamma: \tilde{M}/\langle \gamma \rangle \to X/\langle \rho(\gamma) \rangle.$$ We suppress the $\theta$ from our notation for convenience. $\beta$ projects in the quotient to a core geodesic $\overline{\beta}$. From the $\textrm{CAT}(-1)$ hypothesis, this is the unique geodesic in the homotopy class. Any $D_r/\langle \gamma \rangle$ identifies isometrically with the cylinder $$\{(x,y)=x+iy: 0\leq x\leq \tau, a\leq y \leq r\}$$ with the usual identification. 

\begin{lem}
There is a translation $\tilde{R}$ of the geodesic axis of $\rho(\gamma)$ such that the map $M \ni z \mapsto d(f^\theta,\tilde{R}\circ \varphi^\theta)(z)$ tends to $0$ as we move into the puncture.
\end{lem}

\begin{proof}
We define $\mathcal{C}_\infty$ to be the infinite cylinder $$\{(x,t)\in [0,1]\times (-\infty,\infty) : (0,t)\sim (1,t)\}$$ with the flat metric.
 Let $b_s:\mathcal{C}_\infty\to D/\langle \gamma \rangle$ be the map given by 
 \[ \begin{cases} 
      (x,t)\mapsto (x,s) & -\infty \leq t\leq -s \\
      (x,t)\mapsto (x, 2s+ t) & -s\leq t\leq s \\
      (x,t)\mapsto (x,3s) & s\leq t \leq \infty. 
   \end{cases}
\]
Then set $B_s:=f_\gamma\circ b_s$ and $\varphi_s:=\varphi_\gamma\circ b_s$. Both $B_s$ and $\varphi_s$ are harmonic on $-s\leq t \leq s$ because $b_s$ is conformal there. From Lemma \ref{key} the orbit of any point under $B_s$ remains in a compact set as $s\to \infty$. The uniform energy bounds from Lemma \ref{pole} permit us to construct a subsequence along which both $B_s$ and $\varphi_s$ converge in the $C^\infty$ topology to harmonic maps $f_\infty$ and $\varphi_\infty$ respectively.

Let $h$ denote the harmonic diffeomorphism of the disk whose Hopf differential is $\Phi$. By \cite[Lemma 3.6]{Wolf}, the Jacobian  $J(h)=H(h)-L(h)$ tends to $0$ as we approach the puncture. From Proposition \ref{energy}, $J(f)\to 0$ as well. Therefore, $J(f_\infty)=0$ and necessarily $\textrm{rank} df_\infty\leq 1$ at each point. By equivariance this is rank $1$ in an open set, and by \cite[Theorem 3]{sampson} the image is contained in a geodesic arc. Again by equivariance, the image must then be a closed geodesic arc. There is only one such arc in the quotient, and hence $f_\infty$ maps onto the core geodesic. Lifting $f_\infty$ and $\varphi_\infty$ to maps from $\mathbb{R}^2$ to the axis of $\rho(\gamma)$, $f_\infty$ and $\varphi$ differ by a translation along $\overline{\beta}$. One can justify that last claim by observing that their distance function is a bounded subharmonic function on $\mathbb{R}^2$--hence a constant--and then following the proof of Lemma \ref{unique}. Lifting back to $\tilde{M}$ this means there is a translation $\tilde{R}$ of the geodesic axis such that for any $r>0$, $$d(f^\theta(x,s_m+2t), \tilde{R}\circ \varphi^\theta(x,s_m+2t))=d(b_{s_m}(x,t),R\circ \varphi_{s_m}(x,t))\to 0$$ as $m\to \infty$ for $-r \leq t \leq r$. In particular, the quantities $d(f_\gamma(x,s_{m}), R\circ\varphi_\gamma(x,s_{m}))$ and $d(f_\gamma(x,s_{m+1}), R\circ\varphi_\gamma(x,s_{m+1}))$ are very close to $0$. Since the relevant distance function is subharmonic, its maximum on $$\{(x,t)\in \mathcal{C}_\infty : s_m \leq t \leq s_{m+1}\}$$ is achieved on the boundary. It follows that $$d(f_\gamma(x,t),R\circ\varphi_\gamma(x,t))\to 0$$ as $t\to \infty$. Returning to the universal cover, we conclude that $$d(f^\theta(z),\tilde{R}\circ \varphi^\theta(z))\to 0$$ as we move toward the puncture.
\end{proof}
\begin{prop} \label{compute}
   $\Phi$ has a pole of order $2$ at the cusp with residue $-\Lambda(\theta)\ell(\rho(\gamma))^2/16\pi^2$.
\end{prop}

\begin{proof}
The lemma above shows $$\lim_{s\to \infty} B_s = R\circ \varphi_\gamma$$ in the $C^0$ topology, and along a subsequence in the $C^\infty$ topology. We prove there is no need to pass to a subsequence. Indeed, if we don't have $C^1$ convergence we can pick a subsequence along which our maps are uniformly far from $f_\infty$ in the $C^1$ norm. One can then use the argument above to pass to a subsequence that converges in the $C^\infty$ sense to $S\circ \varphi_\gamma$ for some other rotation $S$. $C^0$ convergence to $R\circ \varphi$ forces $S=R$, which is a contradiction. Continuing inductively gives $C^k$ convergence for any $k$. The Hopf differential of $f$ then converges to $\textrm{Hopf}(\varphi^\theta)$ as we move into the puncture. The result now follows from the computation in Lemma \ref{pole}. 
\end{proof}
\end{subsection}
\begin{subsection}{Uniqueness of harmonic maps}
Let $f_1$ and $f_2$ be two harmonic maps whose Hopf differentials have second order poles and such that the residues have the same complex argument $\nu \in (-\pi,\pi)$.

 \begin{lem} 
There exists an $A_k>0$ such that as $y\to \infty$, the image of $f_k$ remains in an $A_k$-neighbourhood of the geodesic axis of $\rho(\gamma)$.
\end{lem}

\begin{proof}
Let $\beta_y^k$ be the curve $f_k([0,\tau]\times \{y\})$ in the usual coordinates. From Proposition \ref{energy} and \cite[page 516]{Wolf}, the energy density of $f_k$ is uniformly bounded on $M$ in the flat-cylinder metric. This implies $$\ell(\beta_y^k)\leq A$$ for all $y>0$. We argue each $\beta_{y}^k$ becomes trapped close to the geodesic as $y\to \infty$. If not, there is a subsequence $s_j$ tending to $\infty$ and points $f_k(z_j)\in \beta_{s_j}^k$ such that the closest-point projection onto the geodesic, say $y_j$, satisfies $$d(f_k(z_j),y_j)\to \infty$$ Then $$\ell(\beta_{s_j}^k)\geq d(f_k(z_j),f_k(\gamma\cdot z_j))=d(f_k(z_j),\rho(\gamma)f_k(z_j)).$$
The right most term blows up as $j\to \infty$, and this is a clear contradiction. To verify that last statement, note $f_k(z_j)$ accumulates along a subsequence to a point $\xi\in \partial_\infty X$, and since the distance from $\beta_{s_j}^k$ to the geodesic is uniformly bounded below, this is not an endpoint of the geodesic. In particular, the extension of $\rho(\gamma)$ to $\partial_\infty X$ does not fix $\xi$, and hence if $B_{s_j}^k$ is a neighbourhood of $\xi$ in $X\cup \partial_\infty X$, $$d(B_{s_j}^k\cap X,\rho(\gamma)B_{s_j}^k\cap X)\to \infty$$ as $j\to \infty$.
 \end{proof}
 
  Recall the cylinder $\mathcal{C}_\infty$. Let $b_s^k$ be the map $b_s\circ f_\gamma^k:\mathcal{C}_\infty\to X/\langle \rho (\gamma) \rangle$. Since the energy is controlled and it stays close to the geodesic, $b_s^k$ converges along a subsequence to a harmonic map $f_\infty^k$. By the same argument as in the previous subsection $f_\infty^k$ has image in a geodesic and from equivariance this must be the core geodesic $\overline{\beta}$. One can slightly modify an argument as in the previous subsection to check that $a_s^k$ limits to $f_\infty^k$ along the whole sequence in the $C^\infty$ topology. Moreover $f_k$ limits onto the geodesic $\beta$ as we go further into the cusp. 

 \begin{lem}
 The residue of $f_1$ and $f_2$ is the same.
 \end{lem}
 
 \begin{proof}
  Let $\Phi_k := \textrm{Hopf}(f_k)$.  In the computations to follow, we use the flat-cylinder metric on $M$. Let $\gamma_y(x)$ be the curve $x\mapsto x+iy$. From the discussion above, the length of the core geodesic in $X/\langle \rho(\gamma) \rangle$ is $$\lim_{y\to \infty}\ell_g(f_k(\gamma_y)).$$
  There are differentials $\Phi_k'$ such that $$\Phi_k=e^{i\nu}\Phi_k'$$ That is, a differential that differs from $\Phi_k$ by a rotation and whose residue at the cusp is real. The pullback metrics can thus be written $$f_k^*g = e(f_k)\sigma' dz d\overline{z} + e^{i\nu}\Phi_k'+e^{-i\nu}\overline{\Phi_k'}=e(f_k)\sigma' dz d\overline{z} + 2\Re e^{i\nu}\Phi_k'.$$ 
   Writing $\Phi_k'=\phi_k'(z)dz^2$ in a local coordinate we know that in the cylinder $$|\phi_k'| = H(f_k)^{1/2}L(f_k)^{1/2}=H(f_k)\cdot \frac{L(f_k)^{1/2}}{H(f_k)^{1/2}}.$$ From \cite[Proposition 3.8]{Wolf}, in the strip we can write $$\Phi_k = \Big (e^{i\nu}a_{-2}^k+e^{i\nu}O(e^{-A y})\Big )dz^2,$$ where $a_{-2}^k>0$. From Proposition \ref{energy} and \cite[Lemma 3.6]{Wolf}, we also know $$\frac{L(f_k)}{H(f_k)}\to 1$$ as we move into the puncture. The length of the core geodesic is therefore
  \begin{align*}
      \lim_{y\to \infty}\ell_g(f_k(\gamma_y)) &= \lim_{y\to \infty} \int_0^\tau ||\dot{\gamma_y}(x)||_{f_k^*g} dx \\
      &= \lim_{y\to \infty}\int_0^\tau \sqrt{e(f_k)\sigma'+2\Re e^{i\nu}\phi'}dx \\
      &= \lim_{y\to \infty}\int_0^\tau \sqrt{H(f_k)(1+L(f_k)/H(f_k))+2\Re e^{i\nu}\phi'}dx \\
      &= \tau\sqrt{2|a_{-2}^k|(1+\cos \nu)}
  \end{align*}
  by the dominated convergence theorem.
  Meanwhile, passing to the quotient $\mathbb{H}/\langle \gamma \rangle$ we know the core geodesic has length $\ell(\rho(\gamma))$. We deduce $$\ell(\rho(\gamma) = \tau\sqrt{2|a_{-2}^k|(1+\cos \nu)}.$$ Since $\nu$ is fixed, $|a_{-2}^k|$ does not depend on $k$.
 \end{proof}
 
 Henceforth put $a_{-2}=a_{-2}^k$ ($k=1,2$).
 
 \begin{remark}
 From above we see that the complex argument $\nu$ is related to the twist angle $\theta$ from the previous subsection by $$\theta = \frac{-\sin v}{1+\cos v}.$$
 \end{remark}

 \begin{lem} \label{laststep}
 The distance function $z\mapsto d(f_1,f_2)(z)$ is bounded.
 \end{lem}
 
 \begin{proof}
 It suffices to bound $d(f_\infty^1,f_\infty^2)$ as then it is constant and we can lift to the universal cover. By \cite[Proposition 3.8]{Wolf} we can express $$\Phi_k = \Big (a_{-2}e^{i\nu}+e^{i\nu}O(e^{-Ay})\Big ) dz^2$$ in the cylinder coordinates, where $a_{-2}$ is real. Thus, upon taking $s\to \infty$, the Hopf differential of $f_k^\infty$ is $a_{-2}e^{i\nu}dz^2$. That is, the Hopf differentials of $f_1^\infty$ and $f_2^\infty$ agree. We denote this differential by $\Phi_0$, and highlight that the $\Phi_0$-metric is nonsingular. Set $$w_0(f_k)=\frac{1}{2}\log H_0(f_k)(z) - \frac{1}{2}\log|\Phi_0(z)|.$$ Here $H_0$ denotes the holomorphic energy in the $\Phi_0$-metric, and analogously for the other quantities. From above it is clear that $J_0(f_k)=0$ so $H_0(f_k)=L_0(f_k)$. From $|\Phi_0|=H_0(f_k)^{1/2}L_0(f_k)^{1/2}$ we see $w_0(f_k)=0$. One can compute $e_0=2\cosh 2(w_0(f_k))$. In a coordinate $z=x+iy$ such that $\Phi_0=dz^2$, $$f_k^*g = (e_0+2)dx^2 + (e_0-2)dy^2=2dx^2.$$ Let $\gamma_h$ and $\gamma_v$ be horizontal and vertical curves for the $\Phi_0$-metric. Explicitly, we mean the tangent vectors for $\gamma_h$, $\gamma_v$ always evaluate under $\Phi_0$ to positive and negative numbers respectively. Then, $$\ell(f_k(\gamma_h))=\int_{\gamma_h}\sqrt{e_0+2}dx \hspace{1mm} , \hspace{1mm} \ell(f_k(\gamma_v))=\int_{\gamma_h}\sqrt{e_0-2}dy$$ and we see $$\ell(f_k(\gamma_h))=2\ell(\gamma_h) \hspace{1mm} , \hspace{1mm} \ell(f_k(\gamma_v))=0.$$ Therefore, if $v_a$ is the tangent vector to the geodesic at a point $a$ then for all points $z$, $(df_k)_z(\partial_x) = 2v_{f_k(z)}$ and $(df_k)_z(\partial_y)=0$. In particular, $f_k$ is a constant speed map onto the geodesic in the horizontal direction and constant in the vertical direction. Any two such maps differ by a translation. This establishes the result.
 \end{proof}
 
We apply Lemma \ref{unique} to obtain the uniqueness portion of Theorem \ref{first}. If $f_1\neq f_2$, which is only possible if $\rho$ stabilizes a geodesic, then $f_2$ may be obtained from $f_1$ by precomposing with a lift of the translation found in Lemma \ref{laststep}. The results in this section constitute the proof of Theorem \ref{first}.
\end{subsection}
\end{section}
\begin{section}{The proof of Theorem 1.2}
\begin{subsection}{Non-reductive representations}
When $\rho$ is not reductive we can still produce a harmonic map that will be relevant to the domination problem. The content of the following exposition is contained in \cite{DT} and \cite{GK}. Assume $\rho$ fixes a point $\xi$ on $\partial_\infty X$. Given any geodesic ray $\eta:[0,\infty)\to X$ with an endpoint on $\partial_\infty X$, the \textit{Busemann function} $\beta_\eta:X\to \mathbb{R}$ is defined by $$\beta_\eta(x) = \lim_{t\to \infty} (d(\eta(t),x)-t).$$ The fact that this is well-defined and continuous is standard \cite{Bridson}. Now assume the endpoint is $\xi$. For any isometry $\gamma$ with $\gamma\cdot \xi=\xi$ there is a $m(\gamma)\in \mathbb{R}$ such that $$\beta_\eta(\gamma\cdot x)=\beta_\eta(x)+m(\gamma)$$ and $|m(\gamma)|=\ell(\gamma)$. It is easy to see the function $m\circ \rho :\Gamma\to \mathbb{R}$ is a group homomorphism. Let $\tilde{\eta}$ be any biinfinite oriented geodesic in $\mathbb{H}$ and let $\rho^{red}$ be the representation $\Gamma\to \textrm{PSL}_2(\mathbb{R})$ that acts by translations along $\tilde{\eta}$ with lengths $m\circ\rho$, with signs chosen according to the orientation. Since $\rho^{red}$ stabilizes a geodesic there is a family of equivariant harmonic maps as in Theorem \ref{first}. By construction, for all $\gamma'\in \Gamma$, $$\ell(\rho^{red}(\gamma'))=\ell(\gamma').$$ Consequently, the problem of dominating $\rho$ in length spectrum is equivalent to dominating $\rho^{red}$ in length spectrum. Henceforth if $\rho$ is not reductive we replace it with $\rho^{red}$.
\end{subsection}
\begin{subsection}{Digression: elliptic monodromy}
Looking toward domination, it is necessary to understand the behaviour of a harmonic map $f$ when $\rho$ has elliptic monodromy. In the event $\rho$ has hyperbolic monodromy, the choice of parameter $\theta$ will have no effect here, so we assume $\theta=0$. Let $\xi$ be the point on $\partial_\infty \tilde{M}$ associated to the horocycle for $\gamma$ and let $F$ be the set of points in $X$ fixed by $\rho(\gamma)$.

\begin{prop}
   In the setting above, as $z\to \xi$ the function $f$ limits to an element of $F$. Furthermore $e(f)(z)\to 0$.
\end{prop}

\begin{proof}
 Let $B$ be a relevant horoball for $\gamma$ in the universal cover. By adapting a procedure from \cite[Proposition 4.16]{GK}, we first show that for any choice of $\delta>0$ and $\rho$-equivariant map $w$ that has a uniform Lipschitz constant in $B$ there is a $\rho$-equivariant map $w_\delta$ such that
\begin{itemize}
    \item $w_\delta = w$ on $\tilde{M}\backslash(\Gamma\cdot B)$,
    \item $d(w_\delta(p),w_\delta(q))\leq d(w(p),w(q))$ for all points $p,q\in B$, and 
    \item there is a smaller horoball $B'\subset B$ such that $f_\delta(B')$ is contained in the intersection of the convex hull of $f(B')$ and a ball of radius $\delta$.
\end{itemize}
 Towards this let $\mathcal{D}$ be a fundamental domain for the image of $\partial B$ in the quotient and let $p\in \mathcal{D}$. Let $\pi_t$ be the closest point projection from $B$ onto the closed horoball of distance $t>0$ from $\partial B$ and put $p_t=\pi_t(p)$. Note the map $t\mapsto \pi_t(p)$ is nothing more than the transverse horospherical flow for ($B,\partial B,\xi$). By hyperbolic trigonometry (see \cite[Appendix A]{GK}), $$d(p_t,\gamma\cdot p_t)\to 0$$ as $t\to \infty$.  We next find fundamental domains $\mathcal{D}_t$ of $\pi_t(\partial B)$ containing $p_t$ such that $\textrm{diam}\mathcal{D}_t\to 0$ as $t\to \infty$. By the Lipschitz condition $$d(w(p_t),\rho(\gamma) \cdot w(p_t))\to 0$$ and $\textrm{diam}f(\mathcal{D}_t)\to 0$ as $t\to \infty$. 
 
 Next, there is an $\epsilon(\delta)>0$ such that if $d(x,\rho(\gamma)\cdot x)<\epsilon(\delta)$ then $$d(x,F)<\delta/2.$$ In particular, for $t$ large enough there is a $q_t\in F$ such that $d(w(p_t),q_t)<\delta/2$ and $\textrm{diam}w(\pi_t(\mathcal{D}))<\delta/2$. This implies the $\langle \rho(\gamma) \rangle$-invariant set $w(\pi_t(\partial B))$ is contained in $B(\delta,q_t)$. If $\pi_{B_\delta}$ is the closest point projection onto this ball, then take $w_\delta$ to be the $\rho$-eqiuivariant map that coincides with $w$ on $\tilde{M}\backslash (\Gamma\cdot \pi_t(B))$ and with $\pi_{B_\delta}\circ w$ on $\Gamma\cdot \pi_t(B)$. This has all of the required properties. 
 
  With that cleared up recall that the total energy of $f$ is finite in $B$, so that $\textrm{Hopf}(f)$ has a pole of order at most $1$. By Proposition \ref{energy} and \cite[Proposition 3.13]{Wolf}, $$e(f)\leq A$$ in $B$. From here we make the assumption that $M$ has at least two punctures, around one of which $\rho$ has hyperbolic monodromy. This is the most complicated situation and the other cases are resolved similarly.
  Returning to the sequence $f_r$ from the proof of Proposition \ref{existence}, we may enlarge $A$ if necessary to obtain $$e(f_r)\leq A$$ for all $r$. This guarantees a uniform Lipschitz constant across all $f_r$. Next consider the maps $f_{r,\delta}$. We underline that they agree with $f_r$ on $\partial M^c$ (here we are using the notations and conventions of Proposition \ref{existence}). By definition $$e(f_{r,\delta})\leq e(f_r)$$ everywhere, so that $$E_{M^c}(f_{r,\delta})\leq E_{M^c}(f_r).$$ By the energy minimizing property of harmonic maps this forces $f_{r,\delta}$ to be harmonic. From the finite energy theory, if $\rho$ does not fix a point on the boundary then $f_r=f_{r,\delta}$. If $\rho$ does fix such a point then $f_r$ and $f_{r,\delta}$ differ by a translation along a geodesic. Since they are set to be equal on $\partial M^c$ they agree everywhere. Taking $r\to \infty$ implies $f$ has the listed properties of $f_\delta$.
 
 Now we put $\delta_n=2^{-n}$ and iterate the procedure above. We obtain a sequence of horoballs tending to $\xi$ whose image under $f$ is contained in a closed ball of radius $\delta_n$ that intersects $F$ non-trivially. Taking $n \to \infty$ the first result follows. 
 
 To see that the energy decays to zero, we argue by contradiction: suppose there is a $\delta_0> 0$ and a sequence $z_n$ tending to $\xi$ such that $e(f)(z_n)\geq \delta_0$. Consider the cylinder $$\mathcal{C}=\{(x,t)\in [0,1]\times [0,1] : (0,t)\sim (1,t)\}$$ and take a sequence of conformal embeddings $b_n:\mathcal{C}\to D/\langle \gamma\rangle$ such that the projection of $z_n$ is contained in $\textrm{int}(b_n(\mathcal{C}))$. The energy density of $B_n:=f_\gamma\circ b_n$ is uniformly bounded, and we can choose $b_n$ so that $e(B_n)(x,t)=e(f_\gamma)(b_n(x,t))$. By the usual argument, $B_n$ subconverges in the $C^\infty$ sense to a harmonic map $B_\infty:\mathcal{C}\to X/\langle \rho(\gamma)\rangle$. From the first result, we see $B_\infty$ is constant, which forces a contradiction in that $e(f)(z_n)$ must then tend to $0$.
\end{proof}
\end{subsection}
\begin{subsection}{The proof of Theorem 1.2}
 Take any $\rho$-equivariant harmonic map $ (\tilde{M},\sigma)\to X$ produced by Theorem \ref{first} or a map $\tilde{M}\to \mathbb{H}$ from subsection 6.1 and call it $f$. Let $\Phi$ denote the Hopf differential. By Theorem \ref{frick} there is a surface $(N,\sigma_0)$ with cusps and infinite funnels attached along closed geodesics as well as a harmonic map $h$ with Hopf differential $\Phi$ taking $M$ diffeomorphically onto the interior of the convex core $N$. Lift $h$ to a map between the universal covers, that we will still denote $h$. We will always identify $\tilde{M}$ and the universal cover of $N$ with $\mathbb{H}$. Let $j$ denote the holonomy of $h^*\sigma$. Proposition \ref{energy} implies that $$\psi:=f\circ h^{-1}:(h(\mathbb{H}),\sigma_0)\to (X,g)$$ is $(j,\rho)$-equivariant and $1$-Lipschitz. Indeed, $h^*\sigma\geq f^*g$ in the sense that $$(h^*\sigma)_z(v,v)\geq (f^*g)_z(v,v)$$ for all points $z$ and non-zero vectors $v\in T_z\mathbb{H}$. Hence for any two points $x,y\in \mathbb{H}$ and path $c$ from $x$ to $y$, $$\ell_g(\psi(c))=\ell_{f^*g}(h^{-1}(c))\leq \ell_{h^*\sigma}(h^{-1}(c))=\ell_{\sigma}(c).$$ Now, extend $\psi$ to the lift of the boundary of the convex core via uniform continuity, and precompose $\psi$ with the $(j,j)$-equivariant $1$-Lipschitz nearest point projection from $\mathbb{H}$ to the preimage of the convex core of $\mathbb{H}/j(\Gamma)$. This resulting map from $\mathbb{H}\to X$ is $1$-Lipschitz and $(j,\rho)$-equivariant. This chosen $j$ is $j_M$ from the statement of Theorem \ref{bigguy}.

\begin{remark}
For this construction, it does not matter which initial harmonic map $f$ we choose. Going forward we work with $\theta=0$.
\end{remark}

\begin{lem}
In general $j_M$ does not strictly dominate $\rho$. If $X=\mathbb{H}$, a necessary and sufficient condition is that the image of any peripheral isometry under $\rho$ is elliptic. In the general case, a sufficient condition is that 
$$\limsup_{m\to \infty} \frac{d(\psi(p),\rho(\gamma^m)\psi(p))}{2\log m}<1.$$ This will be achieved if $\rho$ has no hyperbolic or parabolic monodromy, but is still possible with parabolic monodromy.
\end{lem}

\begin{proof}
If $\gamma$ is a peripheral isometry such that $\rho(\gamma)$ is hyperbolic, then the translation length is fully encoded by the residue of the Hopf differential, so that $\ell(j_M(\gamma))=\ell(\rho(\gamma))$. Hence $1$ is the optimal Lipschitz constant in this setting. If $j_M(\gamma)$ is parabolic, then by elementary hyperbolic trigonometry (see \cite[Lemma 2.7]{GK}), $$\ell(j_M(\gamma^m))=2\log m + A.$$ If $\rho(\gamma)$ is parabolic then $$\ell(\rho(\gamma^m))\leq 2\log m + A.$$ This follows from \cite[Theorem 1]{horosphere} and a minor modification of the argument in \cite[Lemma 2.7]{GK}. We have equality above if all sectional curvatures of $X$ are $-1$, which implies $\rho$ cannot have parabolic monodromy if $X=\mathbb{H}$. When the image is elliptic we have shown $e(f)\to 0$ at the cusp. From \cite[Proposition 3.13]{Wolf} we have $e(h)\to 1$ at such a cusp. This handles the case $X=\mathbb{H}$.

Working with arbitrary $X$, from the proof of Proposition \ref{energy} we know $\psi$ either has Lipschitz constant $1$ everywhere or the Lipschitz constant is strictly less than $1$ on every compact set. By equivariance of $\psi$ this implies the condition
 $$\limsup_{m\to \infty} \frac{d(\psi(p),\rho(\gamma^m)\psi(p))}{2\log m}<1$$ is sufficient. To see the last statement, simply fix $\kappa<-1$ and consider a rescaled copy of $\mathbb{H}$ with a metric of constant curvature $\kappa$.
 \end{proof}
 
 With $j_M$ in hand, we perturb it to a convex cocompact representation that strictly dominates $\rho$. Let us first assume $j_M$ is convex cocompact. We use the strip deformations of Thurston, which we now describe. Recall that an \textit{arc} of a complete hyperbolic surface $S$ is any non-trivial isotopy class of complete curves in $S$ such that both ends exit into an infinite funnel. A \textit{geodesic arc} is the geodesic representative of an arc.
   
   \begin{defn}
   A \textit{strip deformation} of a hyperbolic surface $S$ along a geodesic arc $\alpha$ is the new surface obtained by cutting along $\alpha$ and gluing a \textit{strip}, the region on $\mathbb{H}$ bounded between two ultraparallel geodesics. The strip is inserted without shearing: so that the two endpoints of the most narrow cross section are identified to a single point $z$, which is called the \textit{waist}. A strip deformation along a collection of pairwise disjoint and non-isotopic geodesic arcs $\alpha_1,\dots, \alpha_n$ is the hyperbolic surface produced by performing strip deformations along $\alpha_k$ iteratively.
   \end{defn}
   
Note that strip deformations along geodesic arcs commute because the curves are disjoint. It was observed by Thurston \cite{Thurston} and proved in full detail in \cite{pap} that as soon as a collection of pairwise disjoint non-isotopic arcs decomposes $S$ into disks, the corresponding strip deformation uniformly lengthens all closed geodesics.
Any geodesic arc $\alpha$ has two parameters associated to a strip deformation, namely the \textit{waist} and the \textit{width}: the thickness of the strip at its most narrow cross section. The lemma below follows from \cite[Theorem 1.8]{margulis}.

 \begin{lem}
 For any choice of geodesic arcs $(\alpha_1,\dots,\alpha_n)$ that decompose $\mathbb{H}/j_M(\Gamma)$ into disks, as well as waist and width parameters $z_{k}$, $w_{k}$, the holonomy of the corresponding strip deformation strictly dominates $j_M$.
 \end{lem}
 
 \begin{figure}[ht]
     \centering
     \includegraphics{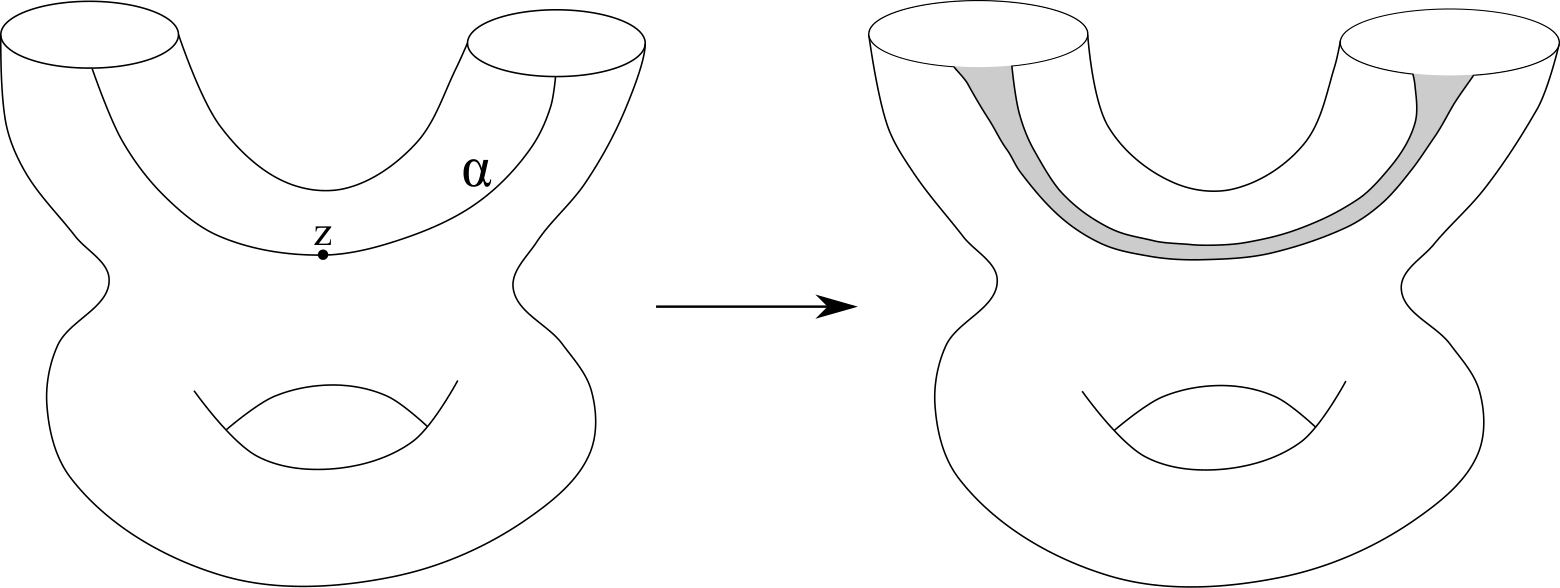}
     \caption{A strip deformation along a geodesic arc on a two-holed torus}
     \label{fig:alien}
 \end{figure}

 \begin{remark}
The complex of \textit{arc systems} $\mathcal{S}$ is the subcomplex of the arc complex obtained by removing all cells that do not divide $S$ into disks.  Danciger, Gu{\'e}ritaud, and Kassel established a homeomorphism between an abstract cone over $\mathcal{S}$ and the subspace of the Fricke-Teichm{\"u}ller space of representations strictly dominating a convex cocompact representation. See \cite{margulis} for the full description.
\end{remark}

This solves the convex cocompact case. Without this condition we proceed as follows. For each puncture $p_k$ in $\mathbb{H}/j_M(\Gamma)$ select disjoint biinfinite geodesic arcs $\alpha_k$ such that both ends of each $\alpha_k$ escape toward $p_k$. Arbitrarily choose points on $\alpha_k$ as a waist parameter and pick some positive width parameters. Insert a hyperbolic strip without shearing, exactly as one would do for a convex cocompact surface. The resulting surface admits a complete hyperbolic metric of infinite area, and therefore its holonomy is convex cocompact (I thank Peter Smillie for pointing this out to me).  The length spectrum of this new holonomy obviously dominates that of $j_M$. Then perform a strip deformation on the new surface to obtain a representation that strictly dominates $j_M$ in length spectrum. By Theorem \ref{gkthm}, in this context length spectrum domination implies domination in the regular sense.

 This completes the proof of Theorem \ref{bigguy} for complete finite volume hyperbolic manifolds. The general case is now a consequence of Lemma \ref{lem:selberg}.
\end{subsection}
\end{section}
\begin{section}{Maximal immersions in the Grassmanian}
We briefly recall some notions from the theory of geometric structures and construct new maximal immersions into $\textrm{Gr}^+(2,4)$, the Grassmanian of timelike planes in $\mathbb{R}^{2,2}$.
Baraglia developed the correspondence between projective structures and maximal immersions in his thesis \cite{bar}, and Alessandrini and Li furthered this connection for AdS structures in \cite{AL}. We are essentially inputting our new structures into their framework.
\begin{subsection}{A rapid review of geometric structures}
 We give our definitions in terms of flat bundles. A \textit{geometry} is a pair $(G,X)$ with $G$ a Lie group and $X$ a manifold endowed with a transitive and effective action of $G$. Given a $C^\infty$ manifold $M=\tilde{M}/\Gamma$ with $\dim M = \dim X$ and a representation $\rho: \Gamma\to G$, we consider the flat bundle $X_\rho$.

\begin{defn}
 A section of $X_\rho$ is \textit{transverse} if the associated $\rho$-equivariant map is a local diffeomorphism.
\end{defn}

\begin{defn}
 A $(G,X)$-\textit{geometric structure} on $M$ is the data of a representation $\rho:\Gamma\to G$ and a transverse section $s:M\to X_\rho$.
\end{defn}
In this section, an $\textrm{AdS}^3$ geometric structure is a  $(G,X)$-geometric structure with $$X=\{x\in \mathbb{R}^4: \mathcal{B}(x,x)=-1 \}$$ (recall the definition of $\mathcal{B}$ from subsection 1.2), and $G=\textrm{SO}_0(2,2)$, which preserves $\mathcal{B}$. Here we identify $\textrm{SO}_0(2,2)$ with $(\textrm{SL}_2(\mathbb{R})\times \textrm{SL}_2(\mathbb{R}))/\mathbb{Z}_2$ (see \cite{AL} for this isomorphism). The AdS $3$-manifolds constructed so far are quotients of $X$ and in this way inherit $\textrm{AdS}^3$-structures.
\begin{remark}
Our conventions now differ from that of subsection 1.2., in that now we're using a double cover of $\textrm{AdS}^3$. We've done this so we don't have to adjust computations from \cite{AL}.
\end{remark}
\end{subsection}
\begin{subsection}{Anti-de Sitter geometric structures}
 Fix a complete finite volume hyperbolic surface $M=\tilde{M}/\Gamma$ and consider two reductive representations $\rho,j:\Gamma \to \textrm{SL}_2(\mathbb{R})$ with $\rho$ arbitrary and $j$ geometrically finite. For concreteness, we assume $M$ has one cusp and $\rho$ has hyperbolic monodromy. This is the most interesting situation and other cases are treated similarly. We view the tensor product $\beta:=j\otimes \rho$ as a representation to $\textrm{SO}_0(2,2)$. $M$ has a canonically determined holomorphic structure, and henceforth we view it simultaneously as a hyperbolic surface and a Riemann surface. 
 
  In \cite{AL}, the only necessary tool for starting their analysis is a Higgs bundle. For a full definition see \cite{Hitchin}. The representations induce flat unimodular vector bundles $(E_k, \nabla_k, \omega_k)$ for $k=1,2$. Here, $E_1=\mathbb{R}^4_\rho$, $E_2=\mathbb{R}^4_j$, $\nabla_k$ are the flat connections and $\omega_k$ are the $\nabla_k$-parallel volume forms. $\beta$ gives its own flat bundle $(E,\nabla, \mathcal{B})$, where $E=\mathbb{R}^4_{\beta}=E_1\otimes E_2$, $\nabla = \nabla_1\otimes \nabla_2$, and $\mathcal{B}=-\omega_1\otimes \omega_2$ is the explicit realization of a symmetric signature-$(2,2)$ $\nabla$-flat billinear form on $E$.  For both $\rho$ and $j$, Theorem \ref{first} guarantees the existence of a unique equivariant harmonic map from $\tilde{M}\to \mathbb{H}$ whose Hopf differential has a residue with fixed complex argument at the cusp. Let $f_1,f_2$ denote any choice of equivariant harmonic maps for $\rho$ and $j$ respectively. Composing with the isomorphism to the symmetric space $$\mathbb{H}\to \textrm{SL}_2(\mathbb{R})/\textrm{SO}_2(\mathbb{R}),$$ we view the maps as Hermitian metrics $H_1$, $H_2$ on the bundles $E_1$,$E_2$. The flat connections decompose uniquely as $$\nabla_i = \nabla_{H_i}+\Psi_{i},$$ where $\nabla_{H_i}$ is an $H_i$-unitary connection and $\Psi_{i}\in \Omega^1(M,\textrm{End}(E))$ is self-adjoint in each fiber. Upon complexifying to bundles $E_i^\mathbb{C}$ we have the further decomposition $$\nabla_i= \nabla_{H_i}^{1,0} + \nabla_{H_i}^{0,1}+\Psi_{i}^{1,0}+\Psi_{i}^{0,1},$$ and the data $(E_i,\nabla_{H_i}^{0,1},\Psi_i^{1,0})$ determines a Higgs bundle.
 
 \begin{remark}
 The Higgs bundles come equipped with an extra \textit{parabolic structure}. See \cite{Simpson} and \cite{mondello} for details. This will not be relevant to our discussion.
 \end{remark}
 
With a Higgs bundle in hand, one can redo the local compuations of \cite{AL} to observe:

 \begin{prop}
    The vector bundle $E$ splits as a $\mathcal{B}$-orthogonal direct sum of rank $2$ sub-bundles $E=F_1\oplus F_2$. $F_1$ is timelike and $F_2$ is spacelike.
 \end{prop}
 
  We define a circle bundle $U$ by $$U=\{v\in F_1 : \mathcal{B}(v,v)=-1\}.$$ We let $p:U\to M$ be the projection map. We have a representation $$\overline{\beta}:= \beta \circ p_* : \pi_1(U) \to \textrm{SO}_0(2,2)$$ that is trivial on the fibers. As in \cite{AL} we construct a geometric structure on $U$ with holonomy $\overline{\beta}$. Set $M_\beta$ to be the sub-bundle of $E$ consisting of vectors $v$ with $\mathcal{B}(v,v)=-1$. Let $M_{\overline{\beta}}$ be the pullback bundle $p^*M_\beta$. Now define $s:U\to M_{\overline{\beta}}$ to be the \textit{tautological section}, the map that reframes each $v\in U$ as an element of $M_{\overline{\beta}}$.

Let $g_k$ be the pullback metric $f_k^*\sigma$ on $M$. We write $g_1>g_2$ to mean $g_1(v,v)>g_2(v,v)$ for all non-zero vectors $v$. Note that $g_1$ is a non-degenerate positive definite metric of constant curvature $-1$, while $g_2$ may degenerate at points. It follows from subsection 6.3 and Theorem \ref{gkthm} that if $g_1>g_2$ then $j$ dominates $\rho$. If $g_1>\lambda g_2$ for some $\lambda<1$, this domination is strict. Again, from the local calculations of \cite{AL} we obtain the following.
\begin{prop}
   If $\rho$ and $j$ are chosen so that $g_2>g_1$ then the tautological section $s$ as defined above is transverse. Therefore, the circle bundle $U$ admits an $AdS^3$ geometric structure with holonomy $\overline{\beta}$.
\end{prop}

In the setting above, keep the current assumptions on $\rho$ and choose $j_M$ so that one can choose the equivariant harmonic maps to have the same Hopf differential.
 
 \begin{prop}
    Let $M$, $\rho$, etc. be as above. There is a representation $j_M$ such that the induced action of $(j_M,\rho)$ on $\textrm{PSL}_2(\mathbb{R})$ is not properly discontinuous but there is a circle bundle $p:U\to M$ with an $AdS^3$ geometric structure and holonomy $(j_M\otimes \rho)\circ p_*$
 \end{prop}
 \end{subsection}
 \begin{subsection}{Maximal immersions in the Grassmanian}
  The bundle $E$ over $M$ has structure group $\textrm{O}(2,2)$. $\textrm{O}(2,2)$ also acts on the Grassmanian $\textrm{Gr}(2,4)$, the space of $2$-planes in $\mathbb{R}^4$, and preserves $\textrm{Gr}^+(2,4)$, the subspace of timelike planes with respect to $\mathcal{B}$. Changing the fiber produces a flat bundle $E(\textrm{Gr}^+(2,4))$ with fiber $\textrm{Gr}^+(2,4)$ and the same structure group. Each circle fiber of $U$ is timelike (the computation is identical to that of \cite[Theorem 7.1]{AL}) and hence gives rise to a $2$-plane in the relevant fiber for $E$. We thus obtain a section of $E(\textrm{Gr}^+(2,4))$, which furnishes a $\beta$-equivariant map $f:\tilde{M}\to \textrm{Gr}^+(2,4)$. 
 
 To study $f$, we put a pseudo-Riemannian metric on $\textrm{Gr}^+(2,4)$. Given $\mathcal{B}$-orthonormal timelike vectors $v_1,v_2$, we choose an orientation of $\mathbb{R}^4$ and spacelike vectors $w_1,w_2$ completing $v_1,v_2$ to a positively oriented $\mathcal{B}$-othonormal basis of $\mathbb{R}^4$. The classical \textit{Pl{\"u}cker embedding} induces a map $\textrm{Gr}^+(2,4)\to \Lambda^2 \mathbb{R}^4$ given by $$\textrm{Span}(v_1,v_2)\mapsto v_1\wedge v_2.$$ The wedge product $(v,w)\mapsto v\wedge w$ defines a billinear form on $\Lambda^2\mathbb{R}^4$ of signature $(3,3)$ that restricts to an $\textrm{O}(2,2)$-invariant pseudo-Riemannian metric of signature $(2,2)$ on the image of $\textrm{Gr}^+(2,4)$. This then pulls back to a metric on the Grassmanian. From the computations in \cite[Theorem 6.2]{AL}, $f$ is harmonic with respect to this metric. A mapping into the Grassmanian is \textit{maximal} if it is harmonic and conformal.
 
 \begin{prop}
    Of the representations found in Theorem \ref{bigguy}, $j_M$ induces a maximal immersion.
 \end{prop}
 
 \begin{proof} 
  If we take $j=j_M$ and choose the two harmonic maps to have the same Hopf differential, then $g_2>g_1$ was proved in Section 6. This implies $s$ is transverse and moreover one can deduce that $f$ is an immersion.
 We can find a covering of $M$ by open sets that identify biholomorphically with the upper half space and such that $U$ trivializes as $\mathbb{H}\times S^1$. This gives coordinates $(z,\theta)$, under which we can write $$f=s\wedge \nabla_{\partial_\theta} s.$$ 
  To check $f$ is maximal, $f$ is conformal if and only if $$\langle \nabla_{\partial_z}f, \nabla_{\partial_z}f \rangle_{\textrm{Gr}^+(2,4)} =0.$$ In local coordinates, $$\langle \nabla_{\partial_z}f, \nabla_{\partial_z}f \rangle_{\textrm{Gr}^+(2,4)} = -8 (\phi(f_2)-\phi(f_1)) dv_{\mathbb{R}^4}$$ \cite[Theorem 6.2]{AL}, where $\phi(f_k)$ is the local coordinate expression for the Hopf differential of $f_k$. This is zero by construction.

For any other dominating representation in this paper, it is clear that the map in question cannot be conformal. For instance if $\rho$ has hyperbolic monodromy, the parameters $\theta_1$, $\theta_2$ are chosen arbitrarily, and $j\neq j_M$ strictly dominates $\rho$, then at the relevant cusp $\textrm{res}(\textrm{Hopf}(f_2))= -\Lambda(\theta_2)\ell(j(\gamma))^2/16\pi^2$ and $\textrm{res}(\textrm{Hopf}(f_1))=-\Lambda(\theta_1)\ell(\rho(\gamma))^2/16\pi^2$ with $\ell(j(\gamma))>\ell(\rho(\gamma))$. If these are to be equal the complex arguments must agree, so we need $\theta_1=\theta_2$. In that case it is clear that $$|\textrm{res}(\textrm{Hopf}(f_2))|>|\textrm{res}(\textrm{Hopf}(f_1))|.$$
\end{proof}

The content of Corollary 1.5 is contained in this section.
 \end{subsection}
\end{section}

\bibliographystyle{plain}
\bibliography{main}

\begin{thebibliography}{10}

\bibitem{AL}
Daniele Alessandrini and Qiongling Li.
\newblock Ad{S} 3-manifolds and {H}iggs bundles.
\newblock {\em Proc. Amer. Math. Soc.}, 146(2):845--860, 2018.

\bibitem{bar}
David Baraglia.
\newblock G2 geometry and integrable systems, 2010.

\bibitem{Bridson}
Martin~R. Bridson and Andr\'{e} Haefliger.
\newblock {\em Metric spaces of non-positive curvature}, volume 319 of {\em
  Grundlehren der Mathematischen Wissenschaften [Fundamental Principles of
  Mathematical Sciences]}.
\newblock Springer-Verlag, Berlin, 1999.

\bibitem{chengyau}
S.~Y. Cheng and S.~T. Yau.
\newblock Differential equations on {R}iemannian manifolds and their geometric
  applications.
\newblock {\em Comm. Pure Appl. Math.}, 28(3):333--354, 1975.

\bibitem{chenglemma}
Shiu~Yuen Cheng.
\newblock Liouville theorem for harmonic maps.
\newblock In {\em Geometry of the {L}aplace operator ({P}roc. {S}ympos. {P}ure
  {M}ath., {U}niv. {H}awaii, {H}onolulu, {H}awaii, 1979)}, Proc. Sympos. Pure
  Math., XXXVI, pages 147--151. Amer. Math. Soc., Providence, R.I., 1980.

\bibitem{MCo}
Michel Coornaert.
\newblock Mesures de {P}atterson-{S}ullivan sur le bord d'un espace
  hyperbolique au sens de {G}romov.
\newblock {\em Pacific J. Math.}, 159(2):241--270, 1993.

\bibitem{Corlette1}
Kevin Corlette.
\newblock Flat {$G$}-bundles with canonical metrics.
\newblock {\em J. Differential Geom.}, 28(3):361--382, 1988.

\bibitem{Corlette2}
Kevin Corlette.
\newblock Archimedean superrigidity and hyperbolic geometry.
\newblock {\em Ann. of Math. (2)}, 135(1):165--182, 1992.

\bibitem{DGK3}
Jeffrey Danciger, Fran\c{c}ois Gu\'{e}ritaud, and Fanny Kassel.
\newblock Fundamental domains for free groups acting on anti--de {S}itter
  3-space.
\newblock {\em Math. Res. Lett.}, 23(3):735--770, 2016.

\bibitem{margulis}
Jeffrey Danciger, Fran\c{c}ois Gu\'{e}ritaud, and Fanny Kassel.
\newblock Margulis spacetimes via the arc complex.
\newblock {\em Invent. Math.}, 204(1):133--193, 2016.

\bibitem{DT}
Bertrand Deroin and Nicolas Tholozan.
\newblock Dominating surface group representations by {F}uchsian ones.
\newblock {\em Int. Math. Res. Not. IMRN}, (13):4145--4166, 2016.

\bibitem{Donaldson}
S.~K. Donaldson.
\newblock Twisted harmonic maps and the self-duality equations.
\newblock {\em Proc. London Math. Soc. (3)}, 55(1):127--131, 1987.

\bibitem{DZ}
Sorin Dumitrescu and Abdelghani Zeghib.
\newblock Global rigidity of holomorphic {R}iemannian metrics on compact
  complex 3-manifolds.
\newblock {\em Math. Ann.}, 345(1):53--81, 2009.

\bibitem{ogtext}
James Eells and Luc Lemaire.
\newblock {\em Selected topics in harmonic maps}, volume~50 of {\em CBMS
  Regional Conference Series in Mathematics}.
\newblock Published for the Conference Board of the Mathematical Sciences,
  Washington, DC; by the American Mathematical Society, Providence, RI, 1983.

\bibitem{ES}
James Eells, Jr. and J.~H. Sampson.
\newblock Harmonic mappings of {R}iemannian manifolds.
\newblock {\em Amer. J. Math.}, 86:109--160, 1964.

\bibitem{nonstandard}
William~M. Goldman.
\newblock Nonstandard {L}orentz space forms.
\newblock {\em J. Differential Geom.}, 21(2):301--308, 1985.

\bibitem{GK}
Fran\c{c}ois Gu\'{e}ritaud and Fanny Kassel.
\newblock Maximally stretched laminations on geometrically finite hyperbolic
  manifolds.
\newblock {\em Geom. Topol.}, 21(2):693--840, 2017.

\bibitem{GKW}
Fran\c{c}ois Gu\'{e}ritaud, Fanny Kassel, and Maxime Wolff.
\newblock Compact anti--de {S}itter 3-manifolds and folded hyperbolic
  structures on surfaces.
\newblock {\em Pacific J. Math.}, 275(2):325--359, 2015.

\bibitem{Wild}
Subhojoy {Gupta}.
\newblock {Harmonic maps and wild Teichm{\"u}ller spaces}.
\newblock {\em arXiv e-prints}, page arXiv:1708.04780, Aug 2017.

\bibitem{GS}
Subhojoy Gupta and Weixu Su.
\newblock Dominating surface-group representations into $\mathrm{PSL}_2
  (\mathbb{C})$ in the relative representation variety, 2020.

\bibitem{ursula}
Ursula Hamenst\"{a}dt.
\newblock Length functions and parameterizations of {T}eichm\"{u}ller space for
  surfaces with cusps.
\newblock {\em Ann. Acad. Sci. Fenn. Math.}, 28(1):75--88, 2003.

\bibitem{horosphere}
Ernst Heintze and Hans-Christoph Im~Hof.
\newblock Geometry of horospheres.
\newblock {\em J. Differential Geom.}, 12(4):481--491 (1978), 1977.

\bibitem{Hitchin}
N.~J. Hitchin.
\newblock The self-duality equations on a {R}iemann surface.
\newblock {\em Proc. London Math. Soc. (3)}, 55(1):59--126, 1987.

\bibitem{cllemma}
J\"{u}rgen Jost.
\newblock {\em Harmonic maps between surfaces}, volume 1062 of {\em Lecture
  Notes in Mathematics}.
\newblock Springer-Verlag, Berlin, 1984.

\bibitem{jostyau}
J\"{u}rgen Jost and Shing-Tung Yau.
\newblock Harmonic maps and group representations.
\newblock In {\em Differential geometry}, volume~52 of {\em Pitman Monogr.
  Surveys Pure Appl. Math.}, pages 241--259. Longman Sci. Tech., Harlow, 1991.

\bibitem{JZ}
J\"{u}rgen Jost and Kang Zuo.
\newblock Harmonic maps of infinite energy and rigidity results for
  representations of fundamental groups of quasiprojective varieties.
\newblock {\em J. Differential Geom.}, 47(3):469--503, 1997.

\bibitem{Ka}
Fanny Kassel.
\newblock Quotients compacts des groupes ultram\'{e}triques de rang un.
\newblock {\em Ann. Inst. Fourier (Grenoble)}, 60(5):1741--1786, 2010.

\bibitem{Klingler}
Bruno Klingler.
\newblock Compl\'{e}tude des vari\'{e}t\'{e}s lorentziennes \`a courbure
  constante.
\newblock {\em Math. Ann.}, 306(2):353--370, 1996.

\bibitem{MK}
Vincent Koziarz and Julien Maubon.
\newblock Harmonic maps and representations of non-uniform lattices of {${\rm
  PU}(m,1)$}.
\newblock {\em Ann. Inst. Fourier (Grenoble)}, 58(2):507--558, 2008.

\bibitem{KR}
Ravi~S. Kulkarni and Frank Raymond.
\newblock {$3$}-dimensional {L}orentz space-forms and {S}eifert fiber spaces.
\newblock {\em J. Differential Geom.}, 21(2):231--268, 1985.

\bibitem{Labourieharm}
Fran\c{c}ois Labourie.
\newblock Existence d'applications harmoniques tordues \`a valeurs dans les
  vari\'{e}t\'{e}s \`a courbure n\'{e}gative.
\newblock {\em Proc. Amer. Math. Soc.}, 111(3):877--882, 1991.

\bibitem{labourie}
Fran\c{c}ois Labourie.
\newblock {\em Lectures on representations of surface groups}.
\newblock Zurich Lectures in Advanced Mathematics. European Mathematical
  Society (EMS), Z\"{u}rich, 2013.

\bibitem{heat}
Fanghua Lin and Changyou Wang.
\newblock {\em The analysis of harmonic maps and their heat flows}.
\newblock World Scientific Publishing Co. Pte. Ltd., Hackensack, NJ, 2008.

\bibitem{Lohkamp}
Jochen Lohkamp.
\newblock Harmonic diffeomorphisms and {T}eichm\"{u}ller theory.
\newblock {\em Manuscripta Math.}, 71(4):339--360, 1991.

\bibitem{minda}
David Minda.
\newblock The strong form of {A}hlfors' lemma.
\newblock {\em Rocky Mountain J. Math.}, 17(3):457--461, 1987.

\bibitem{mondello}
Gabriele Mondello.
\newblock Topology of representation spaces of surface groups in {${\rm
  PSL}_2(\Bbb R)$} with assigned boundary monodromy and nonzero {E}uler number.
\newblock {\em Pure Appl. Math. Q.}, 12(3):399--462, 2016.

\bibitem{omori}
Hideki Omori.
\newblock Isometric immersions of {R}iemannian manifolds.
\newblock {\em J. Math. Soc. Japan}, 19:205--214, 1967.

\bibitem{pap}
Athanase Papadopoulos and Guillaume Th\'{e}ret.
\newblock Shortening all the simple closed geodesics on surfaces with boundary.
\newblock {\em Proc. Amer. Math. Soc.}, 138(5):1775--1784, 2010.

\bibitem{pants}
Mark Pollicott and Polina Vytnova.
\newblock Critical points for the {H}ausdorff dimension of pairs of pants.
\newblock {\em Groups Geom. Dyn.}, 11(4):1497--1519, 2017.

\bibitem{Sa}
Fran\c{c}ois Salein.
\newblock Vari\'{e}t\'{e}s anti-de {S}itter de dimension 3 exotiques.
\newblock {\em Ann. Inst. Fourier (Grenoble)}, 50(1):257--284, 2000.

\bibitem{sampson}
J.~H. Sampson.
\newblock Some properties and applications of harmonic mappings.
\newblock {\em Ann. Sci. \'{E}cole Norm. Sup. (4)}, 11(2):211--228, 1978.

\bibitem{schoenyau}
R.~Schoen and S.~T. Yau.
\newblock {\em Lectures on harmonic maps}.
\newblock Conference Proceedings and Lecture Notes in Geometry and Topology,
  II. International Press, Cambridge, MA, 1997.

\bibitem{Scott}
Peter Scott.
\newblock The geometries of {$3$}-manifolds.
\newblock {\em Bull. London Math. Soc.}, 15(5):401--487, 1983.

\bibitem{Simpson}
Carlos~T. Simpson.
\newblock Harmonic bundles on noncompact curves.
\newblock {\em J. Amer. Math. Soc.}, 3(3):713--770, 1990.

\bibitem{T}
Nicolas Tholozan.
\newblock Dominating surface group representations and deforming closed anti-de
  {S}itter 3-manifolds.
\newblock {\em Geom. Topol.}, 21(1):193--214, 2017.

\bibitem{Thurston}
William~P. {Thurston}.
\newblock {Minimal stretch maps between hyperbolic surfaces}.
\newblock {\em arXiv Mathematics e-prints}, page math/9801039, Jan 1998.

\bibitem{Wan}
Tom Yau-Heng Wan.
\newblock Constant mean curvature surface, harmonic maps, and universal
  {T}eichm\"{u}ller space.
\newblock {\em J. Differential Geom.}, 35(3):643--657, 1992.

\bibitem{wolf1}
Michael Wolf.
\newblock The {T}eichm\"{u}ller theory of harmonic maps.
\newblock {\em J. Differential Geom.}, 29(2):449--479, 1989.

\bibitem{Wolf}
Michael Wolf.
\newblock Infinite energy harmonic maps and degeneration of hyperbolic surfaces
  in moduli space.
\newblock {\em J. Differential Geom.}, 33(2):487--539, 1991.

\end{thebibliography}

\end{document}